\documentclass[10pt,onesite,draft]{article}
\newfam\cyrfam

\usepackage{amssymb,latexsym,amscd,amsmath,amsfonts, amsthm,enumerate}

\newtheorem{theorem}{Theorem}[section]
\newtheorem{lemma}[theorem]{Lemma}
\newtheorem{proposition}[theorem]{Proposition}
\newtheorem{corollary}[theorem]{Corollary}

\newtheoremstyle{definition}
  {6pt}
  {6pt}
  {}
  {}
  {\bfseries}
  {.}
  {.5em}
  {}%

\theoremstyle{definition}
\newtheorem{definition}[theorem]{Definition}

\newtheoremstyle{remark}
  {6pt}
  {6pt}
  {}
  {}
  {\bfseries}
  {.}
  {.5em}
  {}%

\theoremstyle{remark}
\newtheorem{remark}[theorem]{Remark}

\newtheoremstyle{note}
  {6pt}
  {6pt}
  {}
  {}
  {\bfseries}
  {.}
  {.5em}
  {}%

\theoremstyle{note}
\newtheorem{note}[theorem]{Note}

\makeatletter
\renewcommand\@makefntext[1]{%
\setlength\parindent{1em}%
\noindent \makebox[1.8em][r]{}{#1}} \makeatother

\begin{document}
\parskip 4pt
\large \setlength{\baselineskip}{15 truept}
\setlength{\oddsidemargin} {0.5in} \overfullrule=0mm
\def\bfh{\vhtimeb}
\date{}
\title{\bf \large ON THE EULER-POINCARE CHARACTERISTIC AND\\ 
MIXED MULTIPLICITIES OF MAXIMAL DEGREES\\
}
\def\b{\vntime}

\author{
\begin{tabular}{ll}
  Truong Thi Hong Thanh & Duong Quoc Viet  \\
\small thanhtth@hnue.edu.vn  & \small vduong99@gmail.com   \\
\end{tabular}\\
\small Department of Mathematics, Hanoi National University of Education\\
\small 136 Xuan Thuy street, Hanoi, Vietnam\\
}
 \date{}
\maketitle \centerline{
\parbox[c]{13.5cm}{
\small{\bf ABSTRACT:} This paper  defines the
Euler-Poincar\'{e} characteristic of joint reductions of ideals which concerns the maximal  terms in the
Hilbert polynomial; characterizes the positivity of mixed
multiplicities in terms of minimal joint reductions; proves the additivity and other elementary
properties for mixed
multiplicities. The results of the paper together with the results of \cite{TV1} seem to show a natural and nice picture of mixed multiplicities of maximal degrees.}}

 \section{Introduction} \noindent
\noindent
Let $(A, \frak{m})$ be a Noetherian local ring  with maximal ideal
$\mathfrak{m}$ and infinite residue field $k = A/\mathfrak{m}.$  Let $M$
be a finitely generated $A$-module.   Let $J$ be an $\frak
m$-primary ideal, $ I_1,\ldots, I_d$ be ideals of $A.$    Put
\begin{align*}
&\mathbf{e}_i = (0, \ldots,  \stackrel{(i)}{1},  \ldots, 0); {\bf n} =(n_1,\ldots,n_d); {\bf k} =(k_1,\ldots,k_d); {\bf
0}=(0,\ldots,0)\in  \mathbb{N}^{d}; \\
&{\bf 1}=(1,\ldots,1)\in  \mathbb{N}^{d}; |{\bf k}| = k_1 + \cdots + k_d;  \mathrm{\bf I}= I_1,\ldots,I_d;
 \mathbb{I}^{\mathrm{\bf n}}= I_1^{n_1}\cdots I_d^{n_d}.
\end{align*} Denote by $P(n_0, {\bf n}, J, \mathbf{I}, M)$ the Hilbert polynomial of the function
$\ell\Big(\dfrac{J^{n_0}\mathbb{I}^{\bf
n}M}{J^{n_0+1}\mathbb{I}^{\bf n}M}\Big),$ by $\bigtriangleup^{(k_0,\mathrm{\bf
k})}P(n_0, {\bf n}, J,  \mathbf{I}, M)$ the $(k_0,\bf
k)$-difference  of the polynomial $P(n_0, {\bf n}, J, \mathbf{I}, M)$. And one  can
write
$P(n_0, {\bf n}, J, \mathbf{I}, M) = \sum_{(k_0, \mathbf{k})\in \mathbb{N}^{d+1}
}e(J^{[k_0+1]}, \mathbf{I}^{[\mathbf{k}]}; M)\binom{n_0 +
k_0}{k_0}\binom{\mathbf{n}+\mathbf{k}}{\mathbf{k}},$ here
 $\binom{\mathbf{n + k}}{\bf n}= \binom{n_1 + k_1}{n_1}\cdots
\binom{n_d + k_d}{n_d}.$
 \footnotetext{\begin{itemize}  \item[
]This research is funded by Vietnam National Foundation for Science and Technology Development (NAFOSTED) under grant number 101.04.2015.01.\item[
]{\bf Mathematics Subject  Classification (2010):} Primary 13H15.
Secondary 13C15, 13D40,14C17. \item[ ]{\bf  Key words and
phrases:} Mixed multiplicity, Euler-Poincare characteristic;
 Joint reduction. \end{itemize}}
Recall that the original mixed
multiplicity theory studied the mixed multiplicities concerning the terms of highest total degree in the Hilbert
polynomial $P(n_0, {\bf n}, J, \mathbf{I}, M)$, i.e., the mixed multiplicities $e(J^{[k_0+1]}, \mathbf{I}^{[\mathbf{k}]}; M)$ with $ k_0 + |\mathrm{\bf
k}|\; = \deg P(n_0, {\bf n}, J, \mathbf{I}, M)$ (see
e.g. [3$-$10, 14$-$16, 18$-$34]). And in a recent paper \cite{TV1}, one considered a larger class than the class of original
mixed multiplicities concerning the terms of maximal degrees in the Hilbert
polynomial $P(n_0, {\bf n}, J, \mathbf{I}, M).$ There is a fact that
 the condition for $\bigtriangleup^{(k_0,\mathrm{\bf k})}P(n_0, {\bf n}, J,
\mathbf{I}, M)$ to be a constant is equivalent to that the term $e(J^{[k_0+1]}, \mathbf{I}^{[\mathbf{k}]}; M)\binom{n_0 +
k_0}{k_0}\binom{\mathbf{n}+\mathbf{k}}{\mathbf{k}}$ in the polynomial $P(n_0, {\bf n}, J, \mathbf{I}, M)$ satisfies the condition $e(J^{[h_0+1]},
\mathbf{I}^{[\mathbf{h}]}; M)=0$ for all $(h_0, \mathbf{h}) >
(k_0,\mathbf{k})$. And in this case, $\bigtriangleup^{(k_0,\mathrm{\bf k})}P(n_0, {\bf n}, J,
\mathbf{I}, M) = e(J^{[k_0+1]}, \mathbf{I}^{[\mathbf{k}]}; M)$
(see e.g.  \cite[Proposition 2.4 and Proposition 2.9]{TV1}). These facts show the natural appearance of the objects given in
\cite{TV1}
by the following definition.

\begin{definition}[{Definition \ref {de2.0+1}}]
 We call  that $e(J^{[k_0+1]},
\mathbf{I}^{[\mathbf{k}]}; M)$ is  the {\it mixed
multiplicity of
 maximal degrees of $M$ with respect to ideals $J,\mathrm{\bf I}$ of the type
$(k_0+1,\mathrm{\bf k})$} if $e(J^{[h_0+1]},
\mathbf{I}^{[\mathbf{h}]}; M)=0$ for all $(h_0, \mathbf{h}) >
(k_0,\mathbf{k}).$
\end{definition}
One has seen the presence of all these mixed multiplicities in \cite[Example 2.12]{TV1}.
The results of \cite{TV1} showed that many important properties
 of usual mixed multiplicities
  not only are still true but also are stated more natural in the broader  class of the mixed multiplicities of maximal degrees.
However, the additivity and other properties of these multiplicities are not yet known. Recall that
by using the multiplicity formula of Rees modules \cite[Theorem 4.4]{HHRT} (which is a generalized version of \cite[Theorem 1.4]{Ve}), although the
additivity of usual mixed multiplicities was solved in \cite{VT1}, but it seems that this method can not be applied to mixed multiplicities of maximal degrees. This paper first studies the additivity and  characterizes the positivity in terms of joint
reductions for these mixed multiplicities.

Remember that Rees in 1984 \cite{Re} gave the notion of joint reductions and considered  the Euler-Poincar\'{e}
series for joint reductions of ideals of dimension $0,$ and from which he showed that each mixed multiplicity is
the multiplicity of a joint reduction.
\begin{definition}[{Definition \ref{de01}}]
Let $\frak I_i$ be a sequence  consisting $k_i$ elements of $I_i$
for all $1 \le i \le d.$
 Put  ${\bf x} = \frak I_1, \ldots, \frak
 I_d$  and $(\emptyset) = 0_A$. Then ${\bf x}$ is called a sequence of {\it the type $\bf k$} in $\bf I.$ And ${\bf x}$
 is called a
 {\it joint
reduction} of $\bf I$ with respect to $M$ of the type $\bf k$ if
$\mathbb{I}^{\mathbf{n} }M = \sum_{i=1}^d(\frak I_i)
\mathbb{I}^{\mathbf{n} - \mathbf{e}_i}M$ for all large $\bf n.$ A joint reduction of the type $\bf k$
is called a {\it minimal joint reduction} if every sequence of the type $\bf h$ in $\bf I$ with $\bf h < \bf k$
is not
   a joint reduction
of $\mathbf{I}$ with respect to $M.$
\end{definition}
\enlargethispage{0.5 cm}

In 1994, by defining the $''$Euler-Poincar\'{e}
polynomial$''$ of joint reductions for  graded modules, Kirby-Rees \cite{KR1}  proved the additive property
and  the additivity and  reduction formula
 for mixed
multiplicities of graded modules. And basing on  the idea  of Serre \cite {Se} and Auslander-Buchsbaum   \cite {AB},
 the authors in \cite{VT3} defined the Euler-Poincar\'{e}
characteristic of mixed multiplicity systems for graded modules.
However, the results obtained from these works  seem to be not close enough to use in proving even the properties for mixed multiplicities of ideals as in \cite{VT1, VT2}.
\enlargethispage{0.5cm}

The above facts are a motivation to encourage us to find a more
specific invariant that first becomes
an effective tool for proving properties of mixed multiplicities of ideals in Noetherian local rings.

 Let ${\bf x} = \mathfrak{I}_1, \ldots,
\mathfrak{I}_d, \mathfrak{I}_0$  be a joint reduction of
$\mathbf{I}, J$ with respect to $M$ of the type $({\bf k},
k_0+1)$, here $ \mathfrak{I}_0 \subset J,  \mathfrak{I}_i \subset
I_i$ $(1 \le i \le d)$ and
  $t_0, t_1, \ldots, t_d$ be  variables over $A.$  Set $X=
\mathfrak{I}_1t_1, \ldots, \mathfrak{I}_dt_d, \mathfrak{I}_0t_0$ and $T^{\bf
  n} = t_1^{n_1}\cdots t_d^{n_d}.$ Rees in \cite{Re} built the Euler-Poincar\'{e}
series basing on the Koszul complex of the module $\bigoplus_{n_0\geqslant0,
  \bf n\geqslant 0}J^{n_0}\mathbb{ I}^{\mathrm{\bf n}}Mt_0^{n_0}T^{\bf
  n}$ with respect to $X.$ Now basing on Kirby and Rees's approach \cite{KR1} in a new context of mixed multiplicities of maximal degrees, we consider the Koszul complex of the module ${\cal M} = \bigoplus_{n_0\geqslant0,
  \bf n\geqslant 0}\mathbb{ I}^{\mathrm{\bf n}}Mt_0^{n_0}T^{\bf
  n}$  with respect to $X,$ and
  build an invariant called the Euler-Poincar\'{e}
characteristic as follows.

    For any $0 \le i \le n = k_0 + |{\bf k}|+1,$ denote by $H_i({X},{\cal M})$
the $i$th homology module of the Koszul complex of  ${\cal M}$ with
respect to ${X}.$ Using results of Kirby and
Rees in \cite{KR1}, we obtain Lemma \ref{le2.8}, and  get by Proposition \ref{thm2.11} and  Lemma \ref{le2.8} that
the sum $\sum_{i=0}^n (-1)^i \ell_A\big(H_i({ X},{\cal
M})_{{(n_0,{\bf n})}}\big)$ is a constant for all large enough $n_0, \bf n.$ And this constant is called
{ \it the Euler-Poincar\'{e}
characteristic of} the joint reduction $\mathbf{x}$ of the module $M$ with respect to the ideals $J, \mathbf{I}$, and denoted by
 $\chi(\mathbf{x}, J, \mathbf{I}, M).$ And our goal is completed by Proposition \ref{le4.2} which shows that  this invariant is additive on $A$-modules and characterizes mixed multiplicities of maximal degrees.

Next, to  state  the  main result, we need to explain more the relationship between some objects. The mixed
multiplicity of  maximal degrees of $M$ with respect to ideals $J,\mathrm{\bf I}$ of
the type $(k_0+1,\mathrm{\bf k})$ is defined if and only if $\bigtriangleup^{(k_0,\mathrm{\bf k})}P(n_0, {\bf n}, J,
\mathbf{I}, M)$ is a constant. This is equivalent to that
there exists  a joint
reduction    of  $\mathbf{I}, J$ with respect to $M$ of the type
$(\mathrm{\bf k}, k_0+1)$ by Proposition \ref{thm2.11} (see Remark \ref{re2.00}).

The paper not only  proves the additivity, but more noticeably characterizes the positivity in terms of minimal  joint
reductions for mixed multiplicities of maximal degrees. And as
one might expect, we obtain the following theorem.
{\setlength{\baselineskip}{13 truept}
\begin{theorem}[{Theorem \ref{thm1.vt}}]\label{thm1.2}
Let $N$ be an $A$-submodule of $M.$ Assume that  the mixed multiplicity  of maximal degrees of $M$ with respect to
 $J, \bf I$ of the type $(k_0+1, {\bf k})$ is defined and let   $\bf x$ be a
joint reduction of the type $(\mathbf{k}, k_0+1)$ of $\mathbf{I}, J$ with respect to $M$. Then $\chi(\mathbf{x}, J, \mathbf{I}, M)$ is independent of $\bf x$ and we have
\begin{enumerate}[{\rm (i)}]
\item $e(J^{[k_0+1]}, \mathbf{I}^{[\mathbf{k}]}; M) =  \bigtriangleup^{(k_0,\mathrm{\bf
k})}P(n_0, {\bf n}, J,  \mathbf{I}, M)=\chi(\mathbf{x}, J, \mathbf{I}, M).$
\item $e(J^{[k_0+1]}, \mathbf{I}^{[\mathbf{k}]}; M) > 0 $ if and only if $\mathbf{x}$ is a minimal joint reduction.
\item $e(J^{[k_0+1]}, \mathbf{I}^{[\mathbf{k}]}; M) = e(J^{[k_0+1]}, \mathbf{I}^{[\mathbf{k}]}; N) +
e(J^{[k_0+1]}, \mathbf{I}^{[\mathbf{k}]}; M/N).$
\end{enumerate}
\end{theorem}}
Theorem \ref{thm1.2} and its consequences together with the results of \cite{TV1} seem to show a natural and pleasant picture of mixed multiplicities of maximal degrees.

 Using Theorem \ref{thm1.2} (iii), we prove the additivity and
reduction formulas (Corollary
\ref{co4.4a}), and formulas concerning the rank of modules (Corollary
\ref{co4.4}) for mixed
multiplicities of maximal degrees. Applying these results, we immediately recover results on mixed multiplicities in \cite{VT1, VT2}. Moreover, from Theorem \ref{thm1.2} (ii), we  get  a characterization  for the positivity of the usual mixed
multiplicities of ideals in terms of minimal joint reductions (see Remark \ref{re4.1}).

This paper is divided into four  sections. Section 2 is devoted to the discussion of  joint reductions  of ideals and differences of Hilbert polynomials and answers questions: when differences of Hilbert polynomials are constants (Prop.  \ref{thm2.11}) and  are positive constants (Prop. \ref{co4.a}).
Section 3 builds  the Euler-Poincar\'{e} characteristic of joint reductions which corresponds with mixed multiplicities of ideals; studies the positivity, the additivity (Prop.  \ref{le4.2}) of this invariant.
In Section 4, we prove the main theorem and corollaries for mixed multiplicities of ideals.

\section{Joint reductions and differences of Hilbert polynomials}

In this section, in terms of joint reductions, we answer two questions: when differences of Hilbert polynomials are constants (see Proposition \ref{thm2.11}) and when differences of Hilbert polynomials are positive constants (see Proposition \ref{co4.a}).

 Let $\mathbf{k},\mathbf{m}, \mathbf{n}\in \mathbb{N}^d.$ Write $ \mathbf{m} > \mathbf{n}$ if $\mathbf{m}- \mathbf{n}\in \mathbb{N}^d $ and there exists $1 \le i \le d$ such that $m_i > n_i$. In addition, for any numerical polynomial $f(\bf n)$ in $\mathbf{n}\in \mathbb{N}^d$, denote by $\bigtriangleup^{\mathbf{k}}f(\bf n)$ the $\bf
k$-difference  of  $f(\bf n)$, which is defined by  $\bigtriangleup^{\mathbf{0}}f({\bf n}) = f({\bf n})$;
 $\bigtriangleup^{\mathbf{e}_i}f({\bf n}) = f({\bf n} ) - f({\bf n}-  {\bf e}_i)$
and
$\bigtriangleup^{\mathbf{k}}f({\bf n}) =
\bigtriangleup^{\mathbf{e}_i}(\bigtriangleup^{\mathbf{k}- \mathbf{e}_i}f({\bf n}))
 $ for $ \mathbf{k} \ge \mathbf{e}_i.$

We assign $\dim M= -\infty$ for $M= 0$ and
the degree $-\infty$ to the zero polynomial. Set $I  =
I_1\cdots I_d; \overline {M}= M/0_M: I^\infty; q=\dim \overline
{M}.$ Recall that the Hilbert function
$\ell\Big(\dfrac{J^{n_0}\mathbb{I}^{\bf
n}M}{J^{n_0+1}\mathbb{I}^{\bf n}M}\Big)$ is a  polynomial of total
degree $q-1$ for all large $n_0, \bf n$ by \cite[Proposition
3.1]{Vi} (see \cite{MV}). Denote
this Hilbert polynomial  by $P(n_0, {\bf n}, J, \mathbf{I}, M)$.

\begin{remark}\label{re2.00a} We have  $\deg P(n_0, {\bf n}, J, \mathbf{I}, M) = q-1.$
 Put
$n_1 = \cdots = n_d = m$ and fix large enough $m$ such that $ P(n_0,
m{\bf 1}, J, \mathbf{I}, M) =
\ell_A\Big(\frac{J^{n_0}I^mM}{J^{n_0+1}I^mM}\Big)$ for  large enough $n_0$ and $0_M:I^\infty = 0_M:I^m.$
 Then $\dim I^mM = \dim \overline{M}$ and $ P(n_0,
m{\bf 1}, J, \mathbf{I}, M)$ is a
polynomial in $n_0$ of degree $\dim \overline{M} -1.$
So  $q-1 = \deg P(n_0,{\bf n},
J, \mathbf{I}, M)= \deg P(n_0, m{\bf 1}, J, \mathbf{I}, M).$
Hence  $\bigtriangleup^{(k_0,\mathbf{0})}P(n_0, {\bf n}, J, \mathbf{I}, M)= \bigtriangleup^{k_0}P(n_0,
m{\bf 1}, J, \mathbf{I}, M) $  if one of the sides is a constant. From this it follows that $\bigtriangleup^{(1,\mathbf{0})}P(n_0, {\bf n}, J, \mathbf{I}, M)= 0$ if and only if $P(n_0, {\bf n}, J, \mathbf{I}, M)$ is a constant, and it also follows that  $\bigtriangleup^{(k_0,\mathbf{0})}P(n_0, {\bf n}, J, \mathbf{I}, M)$ is a constant if and only if $k_0 \ge \deg P(n_0, {\bf n}, J, \mathbf{I}, M)= \dim M/0_M:I^\infty-1.$
\end{remark}

The concept of joint reductions of $\mathfrak{m}$-primary ideals
was given by Rees \cite{Re} in 1984 and was extended to
the set of arbitrary ideals by \cite{Oc, Vi2,  Vi4,  VDT, {VT4}}. In this paper we use this concept  in  another convenient presentation (see \cite{VDT}).

\begin{definition}[see \cite{Re, VDT}] \label{de01}
Let $\frak I_i$ be a sequence  consisting $k_i$ elements of $I_i$
for all $1 \le i \le d.$
 Put  ${\bf x} = \frak I_1, \ldots, \frak
 I_d$  and $(\emptyset) = 0_A$. Then ${\bf x}$ is called a {\it joint
reduction} of $\bf I$ with respect to $M$ of the type $\bf k$ if
$\mathbb{I}^{\mathbf{n} }M = \sum_{i=1}^d(\frak I_i)
\mathbb{I}^{\mathbf{n} - \mathbf{e}_i}M$ for all large $\bf n.$
 A joint reduction of the type $\bf k$
is called a {\it minimal joint reduction} if every sequence of the type $\bf h$ in $\bf I$ with $\bf h < \bf k$
is not
   a joint reduction
of $\mathbf{I}$ with respect to $M.$
 \end{definition}

And a useful tool used in this paper is the concept of weak-(FC)-sequences which was defined  in
\cite{Vi} (see e.g. \cite{DMT, MV, DQV, VDT}) as the following definition.

\begin{definition}[\cite{Vi}] \label{dn23}   An element $x \in A$ is called a
{\it weak}-(FC)-{\it element} of $\bf I$ with respect to $M$ if
there exists $1 \leqslant i \leqslant d$ such that $x \in I_i$ and
the following conditions are satisfied:
 \begin{enumerate}[(FC1):]
 \item $x{M}\bigcap \mathbb{I}^{\mathbf{n}}{M}
= x\mathbb{I}^{\mathbf{n}-\mathbf{e}_i}{M}$ for all large $\bf n.$

\item $x$ is an $I$-filter-regular element with respect to $M,$
i.e.,\;$0_M:x \subseteq 0_M: I^{\infty}.$
 \end{enumerate}

Let $x_1, \ldots, x_t$ be elements of $A$. For any $0\leqslant i
\leqslant t,$ set\;      $M_i = {M}\big/{(x_1, \ldots, x_{i})M}$.
Then $x_1, \ldots, x_t$ is called  a {\it weak}-(FC)-{\it
sequence}   of $\mathbf{I}$ with respect to $M$ if $x_{i + 1}$ is
a weak-(FC)-element  of $\mathbf{I}$ with respect to  $M_i$ for
all $0 \leqslant i \leqslant t - 1$. If a weak-(FC)-sequence  of
$\mathbf{I}$ with respect to $M$ consists of $k_1$ elements of
$I_1,\ldots,k_d$ elements of  $I_d$ ($k_1,\ldots,k_d \ge 0$), then
it is called  a weak-(FC)-sequence  of $\mathbf{I}$ with respect
to $M$ of the type $\mathbf{k}$.  A
weak-(FC)-sequence $x_1, \ldots, x_t$ is called a {\it maximal
weak}-(FC)-{\it sequence} if $I\nsubseteq
\sqrt{\mathrm{Ann}({M_{t-1})}}$ and $I\subseteq
\sqrt{\mathrm{Ann}(M_t)}.$
\end{definition}

The following  note recalls some  important properties  of weak-(FC)-sequences.
\begin{note}\label{note12} Let $\frak I_i$ be a sequence  of  elements of $I_i$
for all $1 \le i \le d$.
Assume that $\frak I_1, \ldots, \frak
 I_d$ is a weak-(FC)-sequence  of $\mathbf{I}$ with respect
to $M$. Then
\begin{equation}\label{vttt1}(\frak I_1, \ldots, \frak
 I_d)M \bigcap \mathbb{I}^{\mathbf{n}}{M}
=  \sum_{i=1}^d(\frak I_i)
\mathbb{I}^{\mathbf{n} - \mathbf{e}_i}M \end{equation} for all large $\bf n$ by \cite[Theorem
 3.4 (i)]{Vi4}. And if
 $x \in I_i$  $(1 \le i \le d)$ is  a weak-{\rm (FC)}-element  of
$\mathbf{I}, J$ with respect to $M,$ then
\begin{equation}\label{vttt2}P(n_0, {\bf n}, J, \mathbf{I}, M/xM)= P(n_0, {\bf n}, J, \mathbf{I}, M)-P(n_0, {\bf n}-\mathbf{e}_i, J, \mathbf{I}, M) \end{equation}  by \cite [(3)]{DV} (or the proof of \cite[Proposition 3.3 (i)]{MV}).
\end{note}

  The relationship between weak-(FC)-sequences and differences of Hilbert polynomials
     is showed by the
following proposition.
\begin{proposition}\label{pro2.4b} Set $J= I_0.$ Let $k_i > 0$,
 $x \in I_i$  $(0 \le i \le d)$ be  a weak-{\rm (FC)}-element  of
$\mathbf{I}, J$ with respect to $M$. Then  $\dim M/xM:I^\infty
\leqslant \dim M/0_M:I^\infty -1$ and
$$\bigtriangleup^{(k_0,\mathrm{\bf k})}P(n_0, {\bf n}, J,
\mathbf{I}, M)  =\begin{cases}
\bigtriangleup^{(k_0,\mathrm{\bf k} -  \mathbf{e}_i)}P(n_0, {\bf n}, J,
\mathbf{I},  M\big/xM) \quad\text{ if }\; 1 \le i \le d\\
\bigtriangleup^{(k_0-1,\mathrm{\bf k})}P(n_0, {\bf n}, J,
\mathbf{I},  M\big/xM)
\quad\;\text{ if }\; i = 0.
\end{cases} $$
\end{proposition}

\begin{proof} By (\ref{vttt2}) in Note \ref{note12}, we obtain
 \begin{equation}\label{vttt}P(n_0, {\bf n}, J, \mathbf{I}, M/xM)=\begin{cases}
\bigtriangleup^{(0,\mathrm{\bf e}_i)}P(n_0, {\bf n}, J,
\mathbf{I}, M) \quad \text{ if }\; 1 \le i \le d\\
\bigtriangleup^{(1,\mathrm{\bf 0})}P(n_0, {\bf n}, J, \mathbf{I},
M) \quad \;\text{ if }\; i = 0.
\end{cases} \end{equation}
 By (\ref{vttt}), it follows that $$\deg P(n_0, {\bf n}, J, \mathbf{I}, M/xM) \le \deg P(n_0, {\bf n}, J, \mathbf{I}, M)-1.$$ Hence
$$\dim M/xM:I^\infty \leqslant \dim M/0_M:I^\infty -1$$ by Remark \ref{re2.00a}.
Also by (\ref{vttt}) we get
$$\bigtriangleup^{(k_0,\mathrm{\bf k})}P(n_0, {\bf n}, J,
\mathbf{I}, M)  =\begin{cases}
\bigtriangleup^{(k_0,\mathrm{\bf k} -  \mathbf{e}_i)}P(n_0, {\bf n}, J,
\mathbf{I},  M\big/xM) \quad\text{ if }\; 1 \le i \le d\\
\bigtriangleup^{(k_0-1,\mathrm{\bf k})}P(n_0, {\bf n}, J,
\mathbf{I},  M\big/xM)
\quad\;\text{ if }\; i = 0.
\end{cases} $$
\end{proof}

By the way, we  have some explanations for choosing  the element $x$ in  (\ref{vttt}).

\begin{remark}\label{re2.00vt} In studying mixed multiplicities, one always needs the equation (\ref{vttt}). To have (\ref{vttt}) one used different sequences: Risler and Teissier in 1973 \cite{Te} used superficial sequences of $\frak m$-primary ideals; Viet in 2000 \cite{Vi} used  weak-(FC)-sequences; Trung in 2001 \cite{Tr2} used $''$bi-filter-regular sequences$''$;  Trung and Verma in 2007 \cite{TV} used $(\varepsilon_1,\ldots,\varepsilon_m)$-superficial sequences. However, \cite[Remark 3.8]{DMT} showed that the sequences used in \cite{Te, Tr2, TV} are weak-(FC)-sequences.
Moreover, \cite[Remark 4.1]{Vi6} and \cite[Theorem 3.7 and Remark 3.8]{DMT} seem
to account well for the minimum of conditions of
weak-(FC)-sequences that is used in the proof of (\ref{vttt}) (see \cite[Remark 3.9]{DMT}). This explains why we need to use weak-(FC)-sequences in this paper.
\end{remark}

The next proposition answers the question when the $(k_0,\mathrm{\bf
k})$-difference  of the Hilbert polynomial is a constant  in terms of joint reductions   and
weak-(FC)-sequences.

{\setlength{\baselineskip}{14 truept}
\begin{proposition}\label{thm2.11}
The following statements are equivalent:
\begin{itemize}
\item[$\mathrm{(i)}$] There exists a weak-{\rm (FC)}-sequence  of
$\mathbf{I}, J$ with respect to $M$ of the type $(\mathrm{\bf k},
k_0+1)$ which is a joint reduction of $M$  with respect to
$\mathbf{I}, J$.
 \item[$\mathrm{(ii)}$] There exists  a joint
reduction  of  $\mathbf{I}, J$ with respect to $M$ of the type
$(\mathrm{\bf k}, k_0+1).$
 \item[$\mathrm{(iii)}$]
$\bigtriangleup^{(k_0,\mathrm{\bf k})}P(n_0, {\bf n}, J,
\mathbf{I}, M)$ is a constant.
  \end{itemize}
\end{proposition}}
\begin{proof}
  (i) $\Rightarrow$ (ii): holds  trivially. (ii) $\Rightarrow$ (i):
Let $\mathbf{x}= x_1, \ldots, x_n $ be a joint reduction  of
$\mathbf{I}, J$ with respect to $M$  of the type $(\mathrm{\bf k},
k_0+1).$ Set $J= I_0.$ Assume that $x_1 \in I_i$ for $0 \le i \le d.$
Now, if $IJ \nsubseteq \sqrt{\mathrm{Ann}(M)},$ then
 by \cite[Lemma 17.3.2]{SH} and
\cite[Proposition 2.3]{VDT}, there exists  a weak-(FC)-element
$y_1 \in I_i$  with respect to $M$  such that $y_1, x_2, \ldots,
x_n $ is also a joint reduction. If $IJ \subset \sqrt{\mathrm{Ann}(M)},$ then it can be verified that
$\mathbf{x}= x_1, \ldots, x_n $ is a
weak-(FC)-sequence by Definition \ref{dn23}. Hence there exists  a
weak-(FC)-sequence $y_1, y_2, \ldots, y_n $ of $\mathbf{I}, J$
with respect to $M$ of the type $(\mathrm{\bf k}, k_0+1)$ such
that $y_1, y_2, \ldots, y_n $ is a joint reduction of $\mathbf{I},
J$ with respect to $M$ of the type $(\mathrm{\bf k}, k_0+1)$ by
induction. (i) $\Leftrightarrow$ (iii): By \cite[Proposition
2.3]{VDT} (see \cite[Remark 1]{Vi}), there exists a
weak-$\mathrm{(FC)}$-sequence $\mathbf{x}= x_1, \ldots, x_n$ of
$\mathbf{I}, J$ with respect to $M$  of the type $(\mathrm{\bf k},
k_0+1)$ with $x_n \in J.$ Then $ P(n_0, {\bf n}, J, \mathbf{I}, M/(\mathbf{x})M)=0$
if and only if  $J^{n_0}\mathbb{I}^{\mathbf{n}}M \subset (\mathbf{x})M$ for all large $n_0, \bf n.$
This is equivalent to that
 $\mathbf{x}$ is a joint reduction by (\ref{vttt1}) in Note \ref{note12}  (see also \cite[Corollary
3.5]{VT4}). On the other hand,  from Proposition \ref{pro2.4b},
we obtain
 $$ P(n_0, {\bf n}, J, \mathbf{I}, M/(\mathbf{x})M)=
\bigtriangleup^{(1, \mathbf{0})}P(n_0, {\bf n}, J, \mathbf{I},
M/(\mathbf{x'})M),$$ here $\mathbf{x'} = x_1,\ldots,x_{n-1}.$
So $\mathbf{x}$ is a joint reduction if and only if
$$\bigtriangleup^{(1, \mathbf{0})}P(n_0, {\bf n}, J, \mathbf{I},
M/(\mathbf{x'})M) =0.$$   This is equivalent to $P(n_0, {\bf
n}, J, \mathbf{I}, M/(\mathbf{x'})M)$ is a constant
 by  Remark \ref{re2.00a}.
 Note that
$$\bigtriangleup^{(k_0,\mathrm{\bf k})}P(n_0, {\bf n}, J,
\mathbf{I}, M) = P(n_0, {\bf n}, J, \mathbf{I},
M/(\mathbf{x'})M)$$ by Proposition \ref{pro2.4b}.
  Hence
$\mathbf{x}$ is a joint reduction of $\mathbf{I}, J$ with respect
to $M$  of the type $(\mathrm{\bf k}, k_0+1)$ if and only if
$\bigtriangleup^{(k_0, \mathbf{k})}P(n_0, {\bf n}, J, \mathbf{I},
M)$ is a constant.
\end{proof}

The existence of joint reductions with respect to modules on an exact sequence is shown by the following result.

\begin{corollary} \label{no3.1} Let $0\longrightarrow N \longrightarrow M\longrightarrow
P\longrightarrow 0$ be
 an exact sequence of $A$-modules.
   If $\bigtriangleup^{(k_0,\mathrm{\bf k})}P(n_0, {\bf n}, J,\mathbf{I}, M)$ is a constant, then so are $\bigtriangleup^{(k_0,\mathrm{\bf k})}P(n_0, {\bf n}, J,
\mathbf{I}, N)$  and  $\bigtriangleup^{(k_0,\mathrm{\bf k})}P(n_0, {\bf n}, J,
\mathbf{I}, P)$.  And if $\bf x$ is a joint reduction  of
${\bf I}, J$ with respect to
 $M,$ then $\bf x$ is also a
joint reduction  of ${\bf I}, J$ with respect to
 $N, P.$
\end{corollary}

\begin{proof}  By Proposition \ref {thm2.11}, $\bigtriangleup^{(k_0,\mathrm{\bf k})}P(n_0, {\bf n}, J,
\mathbf{I}, M)$ is a constant if and only if there
exists a joint reduction
 $\bf x$ of ${\bf I}, J$ with respect to
 $M$ of the type $({\bf k}, k_0 + 1).$ In this case, $\bf x$  is also a
joint reduction  of $ {\bf I}, J$ with respect to  $ N, P$ because
$\mathrm{Ann}(N) \supset \mathrm{Ann}(M)$ and $\mathrm{Ann}(P) \supset
\mathrm{Ann}(M)$ (see e.g \cite[Lemma 17.1.4]{SH}). Hence $\bigtriangleup^{(k_0,\mathrm{\bf k})}P(n_0, {\bf n}, J,
\mathbf{I}, N)$ and $\bigtriangleup^{(k_0,\mathrm{\bf k})}P(n_0, {\bf n}, J,
\mathbf{I}, P)$ are also constants.
\end{proof}

To  characterize the positivity of    the $(k_0,\mathrm{\bf
k})$-difference  of the Hilbert polynomial, first we need to prove the following lemma.

\begin{lemma}\label{no4.3a} We have the following.
\begin{enumerate}[{\rm (i)}]
 \item
$\dim \overline{M} \le 0$ if and only if  $I \subseteq \sqrt{\mathrm{Ann}(M)}.$
 \item $P(n_0, {\bf n}, J,  \mathbf{I}, M) $ is a positive constant if and only
if $\dim\overline{M}=1.$ And in this case,  $I \nsubseteq \sqrt{\mathrm{Ann}(M)}.$
\end{enumerate}
\end{lemma}
\begin{proof}
The proof of (i): Note that  $\dim A/\mathrm{Ann}(M):I^\infty=\dim\overline{M}.$ On the other hand we have $\dim A/\mathrm{Ann}(M):I^\infty \leq 0$ if and only if $A/\mathrm{Ann}(M):I^\infty = 0
$ because $(\mathrm{Ann}(M):
I^\infty): I = \mathrm{Ann}(M): I^\infty.$ So we get (i).
$P(n_0, {\bf n}, J,  \mathbf{I}, M) $ is a positive constant if and only if $\dim\overline{M}-1 = 0$ by Remark \ref{re2.00a}. This is equivalent to
$\dim\overline{M} = 1.$ In this case, $\dim A/\mathrm{Ann}(M):I^\infty= 1.$ Hence $I \nsubseteq \sqrt{\mathrm{Ann}(M)}$ by (i).
 \end{proof}

 Let $(B,\frak{n})$ be an Artinian local ring with  maximal ideal $\frak{n}$ and infinite
 residue field $ B/\frak{n}.$ Let $G=\bigoplus_{\mathrm{\bf n}\in \mathbb{N}^d}G_{\mathrm{\bf
n}}$  be a finitely generated standard $\mathbb{N}^d$-graded
algebra over $B$ (i.e., $G$ is generated over $B$
  by elements of total degree 1) and let $E=\bigoplus_{\mathrm{\bf n}\in \mathbb{N}^d}E_{\bf n}$
  be  a finitely generated $\mathbb{N}^d$-graded $G$-module. Set $G_{++}=\bigoplus_{\mathrm{\bf n}\ge \mathrm{\bf 1}}
G_{\mathrm{\bf n}}$.
    Recall that a homogeneous element $a\in
G$ is called a {\it $G_{++}$-filter-regular element with respect
to $E$}  if $(0_E:a)_{\mathrm{\bf n}}=0$ for  all large
$\mathrm{\bf n}.$ And a sequence  $x_1,\ldots, x_t$ in $G$ is
called   a {\it $G_{++}$-filter-regular sequence with respect to
$E$} if
 $x_i$ is a $G_{++}$-filter-regular element with respect to $E/(x_1,\ldots, x_{i-1})E$ for all
 $1 \le i \le t$ (see e.g. \cite[Definition 2.5]{TV1}).

 Denote by
$P_E(\mathrm{\bf n})$ the Hilbert polynomial of $\ell_B[E_{\mathrm{\bf n}}]$. Then  \cite[Proposition 2.4(i)]{TV1} stated that if $\triangle ^{\mathrm{\bf k}}P_E(\mathrm{\bf n})$ is a constant, then $\triangle ^{\mathrm{\bf k}}P_E(\mathrm{\bf n})\geq 0.$ However, the argument that proved this fact in \cite {TV1}
is incorrect. So we need the following note.

\begin{note}\label{n2} For each ${\bf k}=(k_1,\ldots,k_d)\in \mathbb{N}^d,$ there
exists a $G_{++}$-filter-regular sequence ${\bf x}$ in
$\bigcup_{i=1}^dG_{{\bf e}_i}$ with respect to  $E$ consisting of
$k_i$ elements of $G_{{\bf e}_i}$ for all $1 \le i \le d$ by \cite[Remark 2.6(i)]{TV1} (see \cite [Proposition 2.2 and Note (ii)]
{Vi6}). And then
    $\triangle ^{\mathrm{\bf k}}P_E(\mathrm{\bf n})
= P_{E/{\bf x}E}(\mathrm{\bf n})$ by  \cite[Remark 2.6(ii)]{TV1} (see \cite[Remark 2.6]{Vi6}).  Hence $\triangle ^{\mathrm{\bf k}}P_E(\mathrm{\bf n})= \ell_B[(E/{\bf x}E)_{\mathrm{\bf n}}]$ for all large $\bf n$. So $\triangle ^{\mathrm{\bf k}}P_E(\mathrm{\bf n}) \geq 0$  for all large $\bf n$.
From this it follows that if $\triangle ^{\mathrm{\bf k}}P_E(\mathrm{\bf n})$ is a constant, then $\triangle ^{\mathrm{\bf k}}P_E(\mathrm{\bf n})\geq 0.$ Consequently, we get the proof of \cite[Proposition 2.4(i)]{TV1}.
\end{note}

Return to $\bigtriangleup^{(k_0,\mathrm{\bf
k})}P(n_0, {\bf n}, J,  \mathbf{I}, M).$ Then by Note \ref{n2},
 if  $\bigtriangleup^{(k_0,\mathrm{\bf
k})}P(n_0, {\bf n}, J,  \mathbf{I}, M)$ is a constant, then this constant is  non-negative.
And the following result answers the question when $\bigtriangleup^{(k_0,\mathrm{\bf
k})}P(n_0, {\bf n}, J,  \mathbf{I}, M)$ is a positive constant in terms of minimal joint reductions  and maximal weak-$\mathrm{(FC)}$-sequences.

\begin{proposition}\label{co4.a} Assume that $\bigtriangleup^{(k_0,\mathrm{\bf
k})}P(n_0, {\bf n}, J,  \mathbf{I}, M)$ is a constant. Then the following statements are equivalent:
\begin{itemize}
\item[$\mathrm{(i)}$] $\bigtriangleup^{(k_0,\mathrm{\bf
k})}P(n_0, {\bf n}, J,  \mathbf{I}, M)$
 is positive.
\item[$\mathrm{(ii)}$] Every weak-$\mathrm{(FC)}$-sequence $x_1, \ldots, x_n$ with $x_n \in J$ of $\mathbf{I}, J$ with
respect to $M$ of the type $({\bf k},k_0+1)$ is maximal.
\item[$\mathrm{(iii)}$] There exists a maximal
weak-$\mathrm{(FC)}$-sequence $x_1, \ldots, x_n$ with $x_n
\in J$ of $\mathbf{I}, J$ with respect to $M$ of the type $({\bf
k},k_0+1).$

 \item[$\mathrm{(iv)}$] For any joint
reduction $x_1, \ldots,
x_n$ of $\mathbf{I}, J$ with respect to $M$ of the
type $({\bf k},k_0+1)$  with $x_n \in J,$ then $ x_1, \ldots,
x_{n-1}$ is not  a joint reduction of $\mathbf{I}, J$ with respect to
$M.$ \item[$\mathrm{(v)}$] Every joint reduction  of $\mathbf{I},
J$ with respect to $M$ of the type $({\bf k},k_0+1)$ is minimal.
\item[$\mathrm{(vi)}$] There exists a minimal joint reduction of $\mathbf{I}, J$ with respect to
$M$ of the type $({\bf
k},k_0+1).$

 \end{itemize}
\end{proposition}

In the following proof, joint reductions
satisfying  Proposition \ref{co4.a} (iv) are called temporarily
 {\it proper joint reductions}.

\begin{proof} Note that by
\cite[Proposition 2.3]{VDT} (see \cite[Remark 1]{Vi}), there always exists a
weak-$\mathrm{(FC)}$-sequence ${\bf z}=z_1, \ldots, z_n$ of
$\mathbf{I}, J$ with respect to $M$ of the type $({\bf k},k_0+1),$
here $ z_1, \ldots, z_{|\mathbf{k}| } \subset \mathbf{I}$ and $
z_{|\mathbf{k}|+1}, \ldots, z_n \subset J.$

 (i) $\Rightarrow$ (ii): Let ${\bf
x}=x_1, \ldots, x_n$ with $x_n \in J$ be a weak-$\mathrm{(FC)}$-sequence   of $\mathbf{I}, J$ with
respect to $M$ of the type $({\bf k},k_0+1)$. Then   by Proposition \ref{pro2.4b} we have
$$\bigtriangleup^{(k_0,\mathrm{\bf
k})}P(n_0, {\bf n}, J,  \mathbf{I}, M)=
P(n_0, {\bf n}, J,  \mathbf{I}, M/(x_1, \ldots,
x_{n-1})M).$$  So
$P(n_0, {\bf n}, J,  \mathbf{I}, M/(x_1, \ldots,
x_{n-1})M)$ is a positive constant by (i).
Therefore by Lemma \ref{no4.3a} (ii), we get  $I \nsubseteq
\sqrt{\mathrm{Ann}[{M/(x_1, \ldots, x_{n-1})M}]}$ and $$\dim  M/(x_1, \ldots, x_{n-1})M: I^\infty = 1.$$  Hence $IJ \nsubseteq
\sqrt{\mathrm{Ann}[{M/(x_1, \ldots, x_{n-1})M}]}$ since $J$ is $\frak{m}$-primary. Furthermore
  $$\dim  M/(x_1, \ldots, x_{n})M: I^\infty
\le \dim  M/(x_1, \ldots, x_{n-1})M: I^\infty - 1$$ by  Proposition \ref{pro2.4b}. Hence  $\dim M/(x_1, \ldots, x_{n})M: I^\infty \le0.$
So by Lemma \ref{no4.3a} (i), $I \subset \sqrt{\mathrm{Ann}[{M/(x_1, \ldots, x_{n})M}]}.$ Consequently, $IJ \subset \sqrt{\mathrm{Ann}[{M/(x_1, \ldots, x_{n})M}]}.$
  Therefore  ${\bf x}$ is a maximal
weak-$\mathrm{(FC)}$-sequence.  (ii) $\Rightarrow$ (iii) is clear.
(iii) $\Rightarrow$ (i):  Assume that ${\bf x}=x_1, \ldots, x_{n}$
with $x_n \in J$
 is a maximal weak-$\mathrm{(FC)}$-sequence, then
 $I \nsubseteq
\sqrt{\mathrm{Ann}[{M/(x_1, \ldots, x_{n-1})M}]}.$   So we obtain
$\dim  M/(x_1, \ldots, x_{n-1})M: I^\infty > 0$  by  Lemma \ref{no4.3a} (i).
Hence  by Remark \ref{re2.00a},
$ \deg P(n_0, {\bf n}, J,  \mathbf{I}, M/(x_1, \ldots,
x_{n-1})M) \ge 0.$ Moreover,   by Proposition \ref{pro2.4b},
we have $$\bigtriangleup^{(k_0,\mathrm{\bf
k})}P(n_0, {\bf n}, J,  \mathbf{I}, M)=
P(n_0, {\bf n}, J,  \mathbf{I}, M/(x_1, \ldots,
x_{n-1})M).$$  Thus we get
$\bigtriangleup^{(k_0,\mathrm{\bf
k})}P(n_0, {\bf n}, J,  \mathbf{I}, M) \not=0.$ Therefore
(i) $\Leftrightarrow$ (ii) $\Leftrightarrow$ (iii).

(iii) $\Rightarrow$ (iv): By (i) $\Leftrightarrow$ (iii),
$\bigtriangleup^{(k_0,\mathrm{\bf
k})}P(n_0, {\bf n}, J,  \mathbf{I}, M)  > 0$. Now let
${\bf x}=x_1, \ldots, x_n$ be
  a  joint reduction of $\mathbf{I}, J$ with
respect to $M$ of the type $({\bf k},k_0+1)$ with $x_n \in J.$  If
$\bf x$ is not proper, then $x_1, \ldots, x_{n-1}$ is also
  a  joint reduction of $\mathbf{I}, J$ with
respect to $M$ of the type $({\bf k},k_0).$ By \cite[Lemma 17.3.2]{SH} and
\cite[Proposition 2.3]{VDT},
 there exists a weak-(FC)-sequence $y_1, \ldots, y_{n-1}$ such
that $y_1, \ldots, y_{n-1}$ is a joint reduction of $\mathbf{I},
J$ with respect to $M$ of the type $({\bf k},k_0).$ In this case,
$J^{n_0}\mathbb{ I}^{\mathrm{\bf n}}[M/(y_1, \ldots,
y_{n-1})M]=0$ for all large $n_0, \bf n.$ Hence we have
$P(n_0, {\bf n}, J,  \mathbf{I},  M/(y_1, \ldots,
y_{n-1})M) = 0.$  Recall that by  Proposition \ref{pro2.4b},
$\bigtriangleup^{(k_0,\mathrm{\bf
k})}P(n_0, {\bf n}, J,  \mathbf{I}, M)  =
P(n_0, {\bf n}, J,  \mathbf{I},  M/(y_1, \ldots,
y_{n-1})M).$ So  $\bigtriangleup^{(k_0,\mathrm{\bf
k})}P(n_0, {\bf n}, J,  \mathbf{I}, M) =0,$ we get a contradiction. Hence ${\bf x}$ is
  a proper joint reduction. Thus (iii)
$\Rightarrow$ (iv) is proved. (iv) $\Rightarrow$ (iii): By Proposition \ref{thm2.11}, there exists a weak-(FC)-sequence $x_1, \ldots,
x_{n}$ which is a joint reduction of $\mathbf{I}, J$ with respect
to $M$ of the type $({\bf k},k_0+1)$ with $x_n \in J.$ Then in
this case,  $IJ \subset
 \sqrt{\mathrm{Ann}[M/(x_1, \ldots, x_{n})M]}.$
Now if $\bf x$ is not maximal, then $IJ \subset
 \sqrt{\mathrm{Ann}[M/(x_1, \ldots, x_{n-1})M]}$. Therefore we have $J^{n_0}\mathbb{I}^{\mathbf{n}}M \subset (x_1, \ldots, x_{n-1})M$ for all large $n_0, \bf n.$ Thus by (\ref{vttt1}) in Note \ref{note12}, $x_1, \ldots, x_{n-1}$ is a joint reduction of
$\mathbf{I}, J$ with respect to $M$ of the type $({\bf k},k_0).$
So $\bf x$ is not a proper joint reduction. Consequently  (iii)
$\Leftrightarrow$ (iv).

(vi) $\Rightarrow$ (iv) By (vi), it follows that any sequence of the type $({\bf h}, h_0)<({\bf k},k_0+1)$ of $\mathbf{I}, J$ is not
   a joint reduction
of $\mathbf{I}, J$ with respect to $M$ (see Definition \ref{de01}).
Now assume that $x_1, \ldots,
x_n$ is a joint reduction of $\mathbf{I}, J$ with respect to $M$ of the
type $({\bf k},k_0+1)$  with $x_n \in J.$ Then $ x_1, \ldots,
x_{n-1}$ is a sequence of the
type $({\bf k},k_0)$.  Since $({\bf k},k_0) < ({\bf k},k_0+1)$, $ x_1, \ldots,
x_{n-1}$ is not  a joint reduction of $\mathbf{I}, J$ with respect to
$M$.
(iv) $\Rightarrow$ (v): Let ${\bf x}=x_1,\ldots,x_n$ with $x_n \in
J$ be a joint reduction of $\mathbf{I}, J$ with respect to $M$ of
the type $({\bf k},k_0+1).$ If  ${\bf x}$ is not a minimal joint
reduction, then there exists a joint reduction ${\bf u}$ of $\mathbf{I}, J$ with respect to
$M$ of the type $({\bf
h},h_0)$ with $({\bf
h},h_0)< ({\bf k},k_0+1).$ Then there exists a joint reduction ${\bf y}= y_1,\ldots, y_n$ of $\mathbf{I}, J$ with respect to
$M$ of the type $({\bf
k},k_0+1)$ with ${\bf u} \subsetneq {\bf y}.$
Since ${\bf y}$ is a proper joint
reduction, it follows that $y_n \in {\bf u}.$
Hence ${\bf u}$ has
 the type $({\bf h},t_0+1) <({\bf k},k_0+1)$ since $y_n \in J.$
So $\bigtriangleup^{(t_0,\mathrm{\bf
h})}P(n_0, {\bf n}, J,  \mathbf{I}, M)$ is
a constant by Proposition \ref{thm2.11}. Therefore
$\bigtriangleup^{(k_0,\mathrm{\bf
k})}P(n_0, {\bf n}, J,  \mathbf{I}, M) = 0.$ But by (i) $\Leftrightarrow$ (iv), $\bigtriangleup^{(k_0,\mathrm{\bf
k})}P(n_0, {\bf n}, J,  \mathbf{I}, M) \not= 0$,   which is a contradiction. Hence (iv) $\Rightarrow$ (v).
By Proposition \ref{thm2.11}, there exists a joint reduction of $\mathbf{I}, J$ with respect to
$M$ of the type $({\bf
k},k_0+1).$ So (v) $\Rightarrow$ (vi).
The proof is complete.
 \end{proof}

\section{Euler-Poincar\'{e} characteristic of joint reductions}

In this section, we
build the  Euler-Poincar\'{e} characteristic of joint reductions
of ideals whose properties are showed in  Proposition \ref{le4.2}.

Although in the context of graded modules, Kirby-Rees \cite{KR1} built the Euler-Poincar\'{e} characteristic of joint reductions, but our goal is to build an object called the Euler-Poincar\'{e} characteristic of joint reductions of ideals which is used to prove important  properties  of mixed multiplicities of maximal degrees of ideals.  This reality leads us to choose Rees modules as initial objects. However, to prove properties of this  invariant, we need to use  properties of more general graded modules.
 This is the reason why we begin with the case of the following graded modules.

  Let
$S= \bigoplus_{\mathbf{n} \geq \mathbf{0}}S_{\mathbf{n}}$
 be a finitely generated standard $\mathbb{N}^d$-graded algebra over $A$ (i.e., $S$ is generated over $A$
  by elements of total degree 1); $V=\bigoplus_{\bf n\ge \bf 0}V_{\bf n}$
  be  a finitely generated $\mathbb{N}^d$-graded $S$-module.
  Set  $S_i= S_{{\bf e}_i}$
  for $ 1 \leqslant i \leqslant d,$ $\mathbf{S} = S_1, \ldots, S_d.$
 Let $J$ be an $\mathfrak{m}$-primary ideal of $A$. For
the case of these graded  modules, first we would like  to recall the
concept of joint reductions  (see \cite{KR1}).

\begin{definition}\label{de 6/7} Set $J=S_0.$ Let $\frak R_i$ be a sequence  consisting $k_i$
elements of $S_i$ for all $0 \le i \le d$ and $k_0,\ldots,k_d \ge
0.$
 Put  ${\frak R} = \frak R_1, \ldots, \frak
 R_d, \frak R_0$  and $(\emptyset) = 0_S$.
 Then $\frak R$ is called a {\it joint
reduction} of $\mathbf{S}, J$ with respect to $V$ of the type
$({\bf k}, k_0)$ if
$$J^{n_0}V_\mathbf{n} = (\frak
R_0)J^{n_0-1}V_{\mathbf{n}} +\sum_{i=1}^d(\frak R_i) J^{n_0
}V_{\mathbf{n} - \mathbf{e}_i} \mathrm{\; for \; all \; large}\; n_0, \bf n.$$
\end{definition}

Let $t$ be a variable over $A.$ Then with the notations $S$ and $V$ as the above, we have
the $\mathbb{N}^{d+1}$-graded algebra $\mathcal{S} = S[t] =
\bigoplus_{n_0\geqslant 0,\bf n\geqslant 0}S_{\mathrm{\bf
n}}t^{n_0}$ and the $\mathbb{N}^{d+1}$-graded
$\mathcal{S}$-modules  $\mathcal{V} = \bigoplus_{n_0\geqslant
0,\bf n \geqslant 0}V_{\mathrm{\bf n}}t^{n_0}\;
\text{ and }\;\mathcal{V}' = \bigoplus_{n_0\geqslant 0,\bf n
\geqslant 0}J^{n_0}V_{\mathrm{\bf n}}t^{n_0}.$

Let ${\frak R} = \frak R_1, \ldots, \frak
 R_d, \frak R_0$  be a joint reduction of $\mathbf{S},
J$ with respect to $V$ of the type $({\bf k}, k_0)$, where $
\mathfrak{R}_0 \subset J, \mathfrak{R}_i \subset S_i$ for $1 \le i \le d.$ Put $\frak R_0t = \{at \mid a \in \frak R_0 \};$
$X= \frak R_1, \ldots, \frak
 R_d, \frak R_0t;$  $n  = k_0 + |{\bf k}|.$
 Considering the Koszul complex of
$\mathcal{V}$ with respect to $X:$
$$ 0\longrightarrow K_n(X,\mathcal{V})\longrightarrow K_{n-1}(X,\mathcal{V})\longrightarrow
\cdots\longrightarrow K_1(X,\mathcal{V})\longrightarrow
K_0(X,\mathcal{V}) \longrightarrow 0,$$ one obtain the
 sequence of the homology modules
$H_0({X},\mathcal{V}),
 H_1({X},\mathcal{V}),\ldots,H_n({X},\mathcal{V}).$

To be able to apply \cite[Theorem 4.2]{KR1} in
building the main object of this section, we need the following facts.

\begin{note}\label{note0} Considering \cite[Theorem 4.2]{KR1} in the context of $\frak R$ and $V,$
we see that:
\begin{itemize}
\item[(i)]  \cite[Theorem 4.2]{KR1} requires that  $\frak R$
is a joint reduction  with respect to $S$.  However, the proof of \cite[Theorem 4.2]{KR1} only uses the fact that  $\frak R$ is a joint reduction  with respect to  $V$. So one can apply  \cite[Theorem 4.2]{KR1} with the assumption that $\frak R$
is a joint reduction of  $\mathbf{S}, J$ with respect to $V$.

\item[(ii)] Although \cite[Theorem 4.2]{KR1} includes  an assumption on the lower  bound of $k_0+|{\bf k}|$, but this  assumption is only used  for \cite[Theorem 4.2 (iv)]{KR1}. This means that  \cite[Theorem 4.2 (i), (ii), (iii)]{KR1} are still true without  the assumption on the lower  bound of $k_0+|{\bf k}|$.
\end{itemize}
\end{note}

Now, since $\frak R$
is a joint reduction of  $\mathbf{S}, J$ with respect to $V$,
$(X\mathcal{V}')_{(n_0,{\bf n})}= \mathcal{V}'_{(n_0,{\bf n})} $
for all large  $n_0, \bf n$ (i.e., $ X$ is a joint reduction of $\mathcal{V'}$  as in \cite{KR1}).
On the other hand, since $J$ is $\frak m$-primary,
$\ell_A\big((\mathcal{V}/\mathcal{V}')_{{(n_0,{\bf
n})}}\big) < \infty $ for all   $n_0, {\bf n}.$ So
 $\ell_A\big(H_i({
X},\mathcal{V})_{{(n_0,{\bf n})}}\big) < \infty$ for all large enough
$n_0, \bf n$ and for  all $0\leqslant i \leqslant n$ by \cite[Theorem
4.2 (i)]{KR1} together with Note \ref{note0}. And moreover for all large  enough $n_0,
\bf n,$   $$\sum_{i=0}^n (-1)^i \ell_A\big(H_i({ X},\mathcal{V})_{{(n_0,{\bf
n})}}\big)$$ is a polynomial in $n_0, \bf n$   by \cite[Theorem 4.2  (ii)]{KR1}  together with Note \ref{note0}.  Then  we denote by $
\chi(n_0,\mathbf{n},\frak R, J, V)$ this polynomial.

Consider  the Hilbert function
$\ell_A\big((\mathcal{V}/\mathcal{V}')_{(n_0,{\bf n})}\big) =
\ell_A\big(\frac{V_{\mathbf{n}}}{J^{n_0}V_{\mathbf{n}}}\big).$ Then by
 \cite[Theorem 4.1]{HHRT}, this function
 is a polynomial  for all large $n_0,\mathbf{n}.$
Denote
 by $F(n_0, {\bf n}, J, V)$ this polynomial, and
$\bigtriangleup^{(k_0,\mathbf{k})}F(n_0, {\bf n}, J, V)$ the $(k_0,\bf
k)$-difference of the polynomial $F(n_0, {\bf n}, J, V)$.

Let $z_0, z_1, \ldots, z_d$ be variables and  put  $Z= z_1, \ldots, z_d$, $Z^{\bf n}= z_1^{n_1}\cdots
z_d^{n_d}.$
 One writes $ \sum_{n_0\geqslant 0, {\bf n\geqslant 0}}f_1({n_0,\bf
n})z_0^{n_0}Z^{\mathbf{n}}\sim \sum_{n_0\geqslant 0,{\bf
n\geqslant 0}}f_2(n_0,{\bf n})z_0^{n_0}Z^{\mathbf{n}}$
 if $f_1(n_0,\mathbf{n}) = f_2(n_0,\bf n)$ for all large $n_0,\bf n$.
Set  $$\chi(z_0,Z,\frak R, J, V)= \sum_{n_0\geqslant 0,{\bf n \geqslant 0}}\chi(n_0,\mathbf{n},
 \frak R, J, V)z_0^{n_0}Z^{\bf n},$$
 $$\mathcal{F}(z_0,Z,J, V)= \sum_{n_0\geqslant 0,\bf n \ge \bf 0}F(n_0, {\bf n}, J, V)z_0^{n_0}Z^{\bf n}.$$
Recall that  ${\bf k} = (k_1, \ldots, k_d)$ and ${\frak R}$  is a joint reduction of the type $({\bf k}, k_0).$ Then using \cite[Theorem 4.2 (iii)]{KR1},  together with Note \ref{note0},  we get
\begin{equation}\label{pt}
\chi(z_0,Z,\frak R,J, V) \sim [\prod_{i=0}^d(1-z_i)^{k_i}]\mathcal{F}(z_0,Z,J, V).
\end{equation}
Now by using the expressions as  in \cite[Lemma 2.2]{Re} (or see \cite[pages 223-224]{KR1}), (\ref{pt}) follows that
$ \chi(n_0,\mathbf{n},\frak R, J, V)$ is a constant   if and only if  $\bigtriangleup^{(k_0,\mathrm{\bf k})}F(n_0, {\bf n},J,  V)$ is a constant and in this case, $\chi(n_0,\mathbf{n},\frak R, J, V) =\bigtriangleup^{(k_0,\mathrm{\bf k})}F(n_0, {\bf n}, J,  V).$
So we obtain the following  result which  presents
 the relationship
between   $\chi(n_0,\mathbf{n},\frak R, J, V)$ and $\bigtriangleup^{(k_0,\mathrm{\bf k})}F(n_0, {\bf n}, J,  V)$.

\begin{lemma}\label{le2.8}
Let $\frak R$  be a joint reduction of  $\mathbf{S}, J$ with
respect to $V$ of the type $({\bf k}, k_0)$. Then $
\chi(n_0,\mathbf{n},\frak R, J, V)$ is a constant   if and only if
$\bigtriangleup^{(k_0,\mathrm{\bf k})}F(n_0, {\bf n}, J,  V)$ is a
constant. In this case, we have $\chi(n_0,\mathbf{n},\frak R, J, V)=
\bigtriangleup^{(k_0,\mathrm{\bf k})}F(n_0, {\bf n}, J,  V).$
\end{lemma}

The above facts yield the following remarks which will be used in the next part.

\begin{remark}\label{re2.8a} Set $ \mathbf{k}!= k_1!\cdots
k_d!$ and ${\bf n^k}= n_1^{k_1}\cdots n_d^{k_d}.$  Let $\frak R$ be a  joint reduction  of $\mathbf{S}, J$  with
respect to $V$  of the type $(\mathrm{\bf k}, k_0)$. Then we have
the following.
\begin{enumerate}[{\rm (i)}]
 \item $\chi(n_0,\mathbf{n},\frak R, J, \underline{\;\;}\,)$ is additive on short exact sequence of $S$-modules  because of  the additivity of the length (see e.g.
\cite[Lemma 4.6.5]{BH1}).
\item Suppose that  $\frak R$ also is  a  joint reduction of $\mathbf{S}, J$  with
respect to an $S$-graded modules $U$,  and $U_{\mathbf{n}} \cong_A V_{\mathbf{n} + \mathbf{m}}$ for all large $\mathbf{n}$ and a fixed $\mathbf{m}$, and $\chi(n_0,\mathbf{n},\frak R, J, V)$ is a constant. Then by the above construction  we  have $$\chi(n_0,\mathbf{n},\frak R, J, V)= \chi(n_0,\mathbf{n},\frak R, J, U).$$
\item  Assume that $\chi(n_0,\mathbf{n},\frak R, J, V)$ is a constant.
Then $\bigtriangleup^{(k_0,\mathrm{\bf k})}F(n_0, {\bf n}, J,
V)$ is a constant by Lemma \ref{le2.8}
and hence $\bigtriangleup^{(k_0,\mathrm{\bf
k})}F(n_0, {\bf n}, J,  V) \geqslant 0$  by   Note \ref{n2}.
So $\chi(n_0,\mathbf{n},\frak R, J, V) \geqslant 0$ by Lemma \ref{le2.8}.
\end{enumerate}
\end{remark}

Next,  we will define the  Euler-Poincar\'{e} characteristic of joint
reductions of $A$-modules via applying the above facts for the case that  $S$ is the Rees algebra $\frak R(\mathrm{\bf I}; A)$ and $V$ is the   Rees  module $\frak R(\mathrm{\bf I}; M)$ as follows.

 Recall that $I= I_1 \cdots I_d;\mathrm{\bf I}= I_1,\ldots,I_d$ and
 $\mathbb{I}^{\mathrm{\bf n}}= I_1^{n_1}\cdots I_d^{n_d}.$ Let $t_1, \ldots, t_d$ be  variables over $A$. Set $T=t_1,
\ldots, t_d$ and $T^{\bf n}=t_1^{n_1}\cdots t_d^{n_d}$.
We denote by
$$\frak R(\mathrm{\bf I}; A) =
  \bigoplus_{\bf n\geqslant 0}\mathbb{ I}^{\mathrm{\bf n}}T^{\bf n}\; \text{ and }\; \frak R(\mathrm{\bf I}; M) =
   \bigoplus_{\bf n \geqslant 0}\mathbb{ I}^{\mathrm{\bf n}}MT^{\bf n} $$
  the Rees algebra of $\mathbf{I}$ and  the  Rees module of $\mathbf{I}$ with respect to $M$,
  respectively.  Then $\frak R(\mathrm{\bf I}; A)_{\mathbf{e}_i}=
  I_it_i$ for all $1 \le i \le d$
  and $F(n_0, \mathbf{n}, J, \frak R(\mathrm{\bf I}; M))$ is the Hilbert polynomial of the Hilbert function $\ell_A\big(\frac{
  \mathbb{I}^{\mathbf{n}}M}{J^{n_0}\mathbb{I}^\mathbf{n}M}\big),$ i.e.,
    $$F(n_0, \mathbf{n}, J, \frak R(\mathrm{\bf I}; M)) = \ell_A\Big(\frac{
  \mathbb{I}^{\mathbf{n}}M}{J^{n_0}\mathbb{I}^\mathbf{n}M}\Big)$$ for all large $n_0,
  \mathbf{n}.$ Recall that $P(n_0, \mathbf{n}, J, \mathbf{I}, M)$ is the Hilbert polynomial of the Hilbert function $\ell_A(\frac{J^{n_0}
  \mathbb{I}^{\mathbf{n}}M}{J^{n_0+1}\mathbb{I}^\mathbf{n}M})$.
  Then it is easily seen that for any $(k_0, \mathbf{k})\in \mathbb{N}^{d+1},$
    $$\bigtriangleup^{(k_0+1,\mathrm{\bf
k})}F(n_0, {\bf n}, J,  \frak R(\mathrm{\bf I}; M)) =
\bigtriangleup^{(k_0,\mathrm{\bf k})}P(n_0, {\bf n}, J,
\mathbf{I}, M).$$

\begin{note}\label{note1} Let ${\bf x} = \mathfrak{I}_1, \ldots,
\mathfrak{I}_d, \mathfrak{I}_0$  be a joint reduction of
$\mathbf{I}, J$ with respect to $M$ of the type $({\bf k},
k_0+1)$, where $ \mathfrak{I}_0 \subset J,  \mathfrak{I}_i \subset
I_i$ for $1 \le i \le d.$ Set ${\bf x}_T = \mathfrak{I}_1t_1, \ldots,
\mathfrak{I}_dt_d, \mathfrak{I}_0$
and $\mathbf{I}_T = I_1t_1, \ldots, I_dt_d.$  Then it is easily seen
that $\mathbf{x}_T$ is  a joint reduction of $\mathbf{I}_T, J$
with respect to $\frak R(\mathbf{I}; M)  $ of the type $({\bf k},
k_0+1)$. Since $\bf x$ is a joint reduction of $\mathbf{I}, J$
with respect to $M$ of the type $({\bf k}, k_0+1)$,
$\bigtriangleup^{(k_0,\mathrm{\bf k})}P(n_0, {\bf n}, J,
\mathbf{I}, M)$  is a constant by Proposition \ref{thm2.11}. Hence
$\bigtriangleup^{(k_0+1,\mathrm{\bf k})}F(n_0, {\bf n}, J,  \frak
R(\mathrm{\bf I}; M))$ is a constant. Therefore by Lemma
\ref{le2.8}, $ \chi(n_0,\mathbf{n},\mathbf{x}_T, J, \frak
R(\mathbf{I}; M))$ is a constant. Then we denote this constant
 by $\chi(\mathbf{x}, J, \mathbf{I}, M)$ and call it  {\it the
Euler-Poincar\'{e} characteristic of $\bf x$
  with respect to $J,\bf I$ and $M$ of the type $(k_0+1,\mathrm{\bf
k})$.} These facts yield:
\end{note}

  \begin{remark}\label{re4.1a}  Let $\bf x$ be a joint
reduction of  $\mathbf{I}, J$ with respect to $M$ of the type
$({\bf k}, k_0+1).$ Then we obtain the following.
\begin{enumerate}[(i)]
\item  By  Lemma \ref{le2.8}, it implies that
  $ \chi({\bf x}, J, \mathbf{I}, M) = \bigtriangleup^{(k_0,\mathrm{\bf
k})}P(n_0, {\bf n}, J,  \mathbf{I}, M).$ This  also follows
that   $\chi({\bf x}, J, \mathbf{I}, M)$ only depends on the type
$({\bf k}, k_0+1)$, does not depend on the joint reduction $\bf x.$

  \item Let $m$ be a positive integer. Since $ \frak R(\mathbf{I}; M )_{\mathbf{n} + m\mathbf{1}} \cong_A
   \frak R(\mathbf{I}; I^mM)_{\mathbf{n}}$, it follows that
    $\chi({\bf x}, J, \mathbf{I}, M) = \chi({\bf x}, J, \mathbf{I},
    I^mM)$ by Remark \ref{re2.8a} (ii).
\end{enumerate}
\end{remark}

The additivity on graded $S$-modules of the Euler-Poincar\'{e}
characteristic is proven easily.
But proving the additivity of the Euler-Poincar\'{e}
characteristic $\chi({\bf x}, J, \mathbf{I}, M) $ on
$A$-modules is not simple. However,  our  goal  is completed by the following result.

\begin{proposition}\label{le4.2}
Let $N$ be an $A$-submodule of $M.$  Assume that $\bf x$ is a
joint reduction of $\mathbf{I}, J$ with respect to $M$ of the type $({\bf k}, k_0+1)$. Then
\begin{enumerate}[{\rm (i)}]
\item $\chi(\mathbf{x}, J, \mathbf{I}, M) =  \bigtriangleup^{(k_0,\mathrm{\bf
k})}P(n_0, {\bf n}, J,  \mathbf{I}, M).$
\item $\chi(\mathbf{x}, J, \mathbf{I}, M) > 0$ if and only if $\bf x$ is a minimal joint reduction.
\item $\chi(\mathbf{x}, J, \mathbf{I}, M) = \chi(\mathbf{x}, J,
\mathbf{I}, N) + \chi(\mathbf{x}, J, \mathbf{I}, M/N).$
\end{enumerate}
\enlargethispage{0.5cm}
\end{proposition}
\begin{proof}   (i) follows from Remark \ref{re4.1a} (i) and (ii) follows from (i) and Proposition \ref{co4.a}. The  proof of (iii):
Note that $\bf x$ is also a joint reduction of
$\mathbf{I}, J$ with respect to $N, M/N$ by Corollary
 \ref{no3.1}. First, we prove that if
$M'$ is a submodule of $M$, then
\begin{equation}\label{eq2}\chi(\mathbf{x}, J, \mathbf{I},
M') \leqslant \chi(\mathbf{x}, J, \mathbf{I}, M).\end{equation}
Indeed, since $\bf x$ is a joint reduction of $\mathbf{I}, J$ with
respect to $M,$ it follows that
  $\mathbf{x}_T$ is  a joint reduction of
$\mathbf{I}_T, J$ with respect to $\frak R(\mathbf{I}; M)$ as
mentioned  above. Hence  $\mathbf{x}_T$ is also a joint reduction of
$\mathbf{I}_T, J$ with respect to $\frak R(\mathbf{I}; M)\big/\frak R(\mathbf{I}; M')$.  On the other hand, by Corollary
\ref{no3.1}, $\bf x$ is a joint reduction of $\mathbf{I}, J$ with
respect to $M',$ and so
  $\mathbf{x}_T$ is  a joint reduction of
$\mathbf{I}_T, J$ with respect to $\frak R(\mathbf{I}; M')$.  Consider  the exact sequence of $\frak
R(\mathbf{I}; A)$-modules:
$$0\longrightarrow \frak R(\mathbf{I}; M') \longrightarrow \frak R(\mathbf{I}; M)\longrightarrow \frak
R(\mathbf{I}; M)\big/\frak R(\mathbf{I}; M')\longrightarrow 0.$$
 Since $\chi(n_0,\mathbf{n},\mathbf{x}_T, J, \frak R(\mathbf{I}; M))$ and  $\chi(n_0,\mathbf{n},\mathbf{x}_T, J, \frak R(\mathbf{I}; M'))$ are constants by Note \ref{note1}, it follows by Remark \ref{re2.8a} (i) that $ \chi\big(n_0,\mathbf{n}, \mathbf{x}_T, J, \frak R(\mathbf{I}; M)\big/\frak R(\mathbf{I}; M')\big)$ also is a constant  and
 $$\begin{aligned}\chi(n_0,\mathbf{n},\mathbf{x}_T, J, \frak R(\mathbf{I}; M)) &= \chi(n_0,\mathbf{n},\mathbf{x}_T, J, \frak R(\mathbf{I}; M'))\\ &+
 \chi\big(n_0,\mathbf{n},\mathbf{x}_T, J, \frak R(\mathbf{I}; M)\big/\frak R(\mathbf{I}; M')\big).\end{aligned}$$
 On the other hand, by Remark \ref{re2.8a} (iii),  $\chi\big(n_0,\mathbf{n},\mathbf{x}_T, J, \frak R(\mathbf{I}; M)\big/\frak
 R(\mathbf{I}; M')\big)  \geqslant 0$. Hence $\chi(n_0,\mathbf{n},\mathbf{x}_T, J, \frak R(\mathbf{I}; M)) \geqslant
 \chi(n_0,\mathbf{n},\mathbf{x}_T, J, \frak R(\mathbf{I}; M')).$ So we get (\ref{eq2}), that means $$\chi(\mathbf{x}, J, \mathbf{I}, M) \geqslant
 \chi(\mathbf{x}, J, \mathbf{I}, M').$$
Since $N$ is a submodule of $M,$ by Artin-Rees Lemma, there exists
an integer $k>0$ such that
$\mathbb{I}^{\mathbf{n}+k\mathbf{1}}M\cap N =
\mathbb{I}^{\mathbf{n}}(I^{k}M\cap N)$ for all ${\bf n \ge 0}.$
Fix this integer $k$. Then since
$\mathbb{I}^{\mathbf{n}+k\mathbf{1}}M\cap N =
\mathbb{I}^{\mathbf{n}}(I^{k}M\cap N)$ for all ${\bf n \ge 0}$, we
have
$$\mathbb{I}^{\mathbf{n}}\Big(\dfrac{I^kM}{I^kM\cap N}\Big)=  \dfrac{\mathbb{I}^{\mathbf{n}+k\mathbf{1}}M}
{\mathbb{I}^{\mathbf{n}+k\mathbf{1}}M\cap I^kM\cap N}=
\dfrac{\mathbb{I}^{\mathbf{n}+k\mathbf{1}}M}{\mathbb{I}
^{\mathbf{n}+k\mathbf{1}}M\cap N}=
\dfrac{\mathbb{I}^{\mathbf{n}}(I^kM)}{\mathbb{I}^{\mathbf{n}}(I^kM\cap
N)}$$ for all ${\bf n \ge 0}$. Hence we get the exact sequence
$$0\longrightarrow \frak R(\mathbf{I}; I^kM\cap N) \longrightarrow \frak R(\mathbf{I}; I^kM)
\longrightarrow \frak R\Big(\mathbf{I}; \dfrac{I^kM}{I^kM\cap N}\Big)
\longrightarrow 0.$$ Consequently  by Remark \ref{re2.8a} (i),
$$
\chi(\mathbf{x}, J, \mathbf{I}, I^kM) = \chi(\mathbf{x}, J,
\mathbf{I}, I^kM\cap N) + \chi\Big(\mathbf{x}, J, \mathbf{I},
\dfrac{I^kM}{I^kM\cap N}\Big).
$$
On the other hand, by Remark \ref{re4.1a} (ii), $\chi(\mathbf{x},
J, \mathbf{I}, I^kM) = \chi(\mathbf{x}, J, \mathbf{I}, M)$ and
$$\chi\Big(\mathbf{x}, J, \mathbf{I}, \dfrac{I^kM}{I^kM\cap N}\Big) =
\chi(\mathbf{x}, J, \mathbf{I}, I^k(M/N)) = \chi(\mathbf{x}, J,
\mathbf{I}, M/N).$$ So $ \chi(\mathbf{x}, J, \mathbf{I}, M) =
\chi(\mathbf{x}, J, \mathbf{I}, I^kM\cap N) + \chi(\mathbf{x}, J,
\mathbf{I}, M/N).$ Next, by (\ref{eq2}) and since  $I^kN \subset
I^kM\cap N \subset N,$  it follows that
$$\chi(\mathbf{x}, J, \mathbf{I}, I^kN)\leqslant \chi(\mathbf{x}, J,
 \mathbf{I}, I^kM\cap N) \leqslant \chi(\mathbf{x}, J, \mathbf{I}, N).$$
And by Remark \ref{re4.1a} (ii), we have $\chi(\mathbf{x}, J,
\mathbf{I}, I^kN)= \chi(\mathbf{x}, J, \mathbf{I}, N).$ Hence
$$\chi(\mathbf{x}, J, \mathbf{I}, I^kM\cap N) = \chi(\mathbf{x}, J,
\mathbf{I}, N).$$  Consequently,  we obtain $\chi(\mathbf{x}, J,
\mathbf{I}, M) = \chi(\mathbf{x}, J, \mathbf{I}, N) +
\chi(\mathbf{x}, J, \mathbf{I}, M/N).$
\end{proof}

\section{Mixed multiplicities of ideals}

This section studies mixed multiplicities of maximal degrees. The following  facts  will show  the
effectiveness of our approach in this paper.

Set  $  \mathrm{\bf I}^{[\mathrm{\bf
k}]}= I_1^{[k_1]},
  \ldots,I_d^{[k_d]}; {\bf n^k}= n_1^{k_1}\cdots n_d^{k_d}; \mathbf{k}!= k_1!\cdots
k_d!; |{\bf k}| = k_1 + \cdots + k_d.$   Recall that  $I =  I_1\cdots I_d;$ $\overline {M}= M/0_M: I^\infty$ and $q=\dim \overline {M}$.

If  $I
\nsubseteq \sqrt{\mathrm{Ann}(M)},$ then $\overline {M} \not= 0.$
 Remember that by
\cite[Proposition 3.1]{Vi} (see \cite{MV}) the Hilbert polynomial $P(n_0, {\bf n}, J, \mathbf{I}, M)$ of the Hilbert function
$\ell\Big(\dfrac{J^{n_0}\mathbb{I}^{\bf
n}M}{J^{n_0+1}\mathbb{I}^{\bf n}M}\Big)$ has total
degree $q-1$. And in this section we assume that $I
\nsubseteq \sqrt{\mathrm{Ann}(M)}.$
If  write the terms of total degree $q-1$ of
$P(n_0, {\bf n}, J, \mathbf{I}, M)$ in the form $$\sum_{k_0 + |\mathbf{k}| = q - 1}
e(J^{[k_0+1]},\mathrm{\bf I}^{[\mathrm{\bf k}]};
M)\frac{n_0^{k_0}\mathbf{n}^\mathbf{k}}{k_0!\mathbf{k}!},$$ then
 $e(J^{[k_0+1]},
\mathbf{I}^{[\mathbf{k}]}; M)$ is called   the {\it mixed
multiplicity} (or the {\it original mixed multiplicities})  of $M$ with respect to $J, \bf I$ of the type
$(k_0+1, \bf k)$ (see e.g. \cite{MV, Ve, Vi}).

Now we would like to recall objects mentioned in \cite{TV1}
 which concern the maximal terms in the Hilbert  polynomial.
It is well known that one  can
write
$$P(n_0, {\bf n}, J, \mathbf{I}, M) = \sum_{(k_0, \mathbf{k})\in \mathbb{N}^{d+1}
}e(J^{[k_0+1]}, \mathbf{I}^{[\mathbf{k}]}; M)\binom{n_0 +
k_0}{k_0}\binom{\mathbf{n}+\mathbf{k}}{\mathbf{k}},$$
 where
$\binom{\mathbf{n + k}}{\bf n}= \binom{n_1 + k_1}{n_1}\cdots
\binom{n_d + k_d}{n_d}.$
 Then
$e(J^{[k_0+1]}, \mathbf{I}^{[\mathbf{k}]}; M)$ are integers.

\begin{note}\label{note2}
$\bigtriangleup^{(k_0,\mathrm{\bf
k})}P(n_0, {\bf n}, J,  \mathbf{I}, M)$ is a constant if and only if
 $e(J^{[h_0+1]},
\mathbf{I}^{[\mathbf{h}]}; M)=0$ for all $(h_0, \mathbf{h}) >
(k_0,\mathbf{k}).$ In this case,
$\bigtriangleup^{(k_0,\mathrm{\bf
k})}P(n_0, {\bf n}, J,  \mathbf{I}, M)= e(J^{[k_0+1]}, \mathbf{I}^{[\mathbf{k}]}; M)$ (see \cite[Proposition 2.4 (ii)]{TV1}).
\end{note}

From this fact, one  can consider a larger class than the class of original
mixed multiplicities. This is the reason why \cite{TV1}  gave the following concept.

\begin{definition}\label{de2.0+1} We call that $e(J^{[k_0+1]},
\mathbf{I}^{[\mathbf{k}]}; M)$ is  the {\it  mixed
multiplicity of
 maximal degrees of $M$ with respect   to  ideals  $J,\mathrm{\bf I}$ of the type
$(k_0+1,\mathrm{\bf k})$} if $e(J^{[h_0+1]},
\mathbf{I}^{[\mathbf{h}]}; M)=0$ for all $(h_0, \mathbf{h}) >
(k_0,\mathbf{k}).$
\end{definition}

An example in  \cite{TV1} calculated  all   mixed
multiplicities of
 maximal degrees in a specific case. For any $(k_0, \mathbf{k})\in \mathbb{N}^{d+1}$, it is not true,
in general, that the mixed multiplicity of  maximal degrees of the type $(k_0+1, \bf
k)$  is  defined. However, the definiteness of  mixed
multiplicities of
 maximal degrees  is shown by the following.

\begin{remark}\label{re2.00}  Let  $(k_0, \mathbf{k})\in \mathbb{N}^{d+1}.$
Then the mixed
multiplicity of  maximal degrees of $M$ with respect to ideals $J,\mathrm{\bf I}$ of
the type $(k_0+1,\mathrm{\bf k})$ is defined if and only if $\bigtriangleup^{(k_0,\mathrm{\bf k})}P(n_0, {\bf n}, J,
\mathbf{I}, M)$ is a constant. This is equivalent to that
there exists  a joint
reduction  $\mathbf{x}$  of  $\mathbf{I}, J$ with respect to $M$ of the type
$(\mathrm{\bf k}, k_0+1)$ by Proposition \ref{thm2.11}.
\end{remark}

Then the main result of the paper is the following theorem.

\begin{theorem}\label{thm1.vt} Let $N$ be an $A$-submodule of $M.$ Assume that  the mixed multiplicity  of maximal degrees of $M$ with respect to
 $J, \bf I$ of the type $(k_0+1, {\bf k})$ is defined and let   $\bf x$ be a
joint reduction of the type $(\mathbf{k}, k_0+1)$ of $\mathbf{I}, J$ with respect to $M$. Then $\chi(\mathbf{x}, J, \mathbf{I}, M)$ is independent of $\bf x$ and we have

\begin{enumerate}[{\rm (i)}]
\item $e(J^{[k_0+1]}, \mathbf{I}^{[\mathbf{k}]}; M) =  \bigtriangleup^{(k_0,\mathrm{\bf
k})}P(n_0, {\bf n}, J,  \mathbf{I}, M)=\chi(\mathbf{x}, J, \mathbf{I}, M).$
\item $e(J^{[k_0+1]}, \mathbf{I}^{[\mathbf{k}]}; M) > 0 $ if and only if $\mathbf{x}$ is a minimal joint reduction.

\item $e(J^{[k_0+1]}, \mathbf{I}^{[\mathbf{k}]}; M) = e(J^{[k_0+1]}, \mathbf{I}^{[\mathbf{k}]}; N) +
e(J^{[k_0+1]}, \mathbf{I}^{[\mathbf{k}]}; M/N).$
\end{enumerate}
\end{theorem}
\begin{proof} By Remark \ref{re4.1a}(i),  $\chi(\mathbf{x}, J, \mathbf{I}, M)$ is independent of $\bf x$. By Proposition \ref {le4.2}(i) and Note \ref{note2}, we get (i).  (ii) follows from (i) and Proposition \ref{co4.a}.  By (i) and Proposition \ref {le4.2}(iii), we obtain (iii).
\end{proof}

Next from Theorem \ref{thm1.vt} (i) and (iii),  we immediately get the following.

\begin{corollary}\label{co14.3b} Let $0\longrightarrow N\longrightarrow M \longrightarrow
P\longrightarrow 0$ be an exact sequence of $A$-modules.  Assume that  the mixed multiplicity  of maximal degrees of $M$ with respect to
 $J, \bf I$ of the type $(k_0+1, {\bf k})$ is defined and let   $\bf x$ be a
joint reduction of the type $(\mathbf{k}, k_0+1)$ of $\mathbf{I}, J$ with respect to $M$.  Then
\begin{enumerate}[{\rm (i)}]
\item $e(J^{[k_0+1]}, \mathbf{I}^{[\mathbf{k}]}; M) = e(J^{[k_0+1]},
\mathbf{I}^{[\mathbf{k}]}; N) + e(J^{[k_0+1]},
\mathbf{I}^{[\mathbf{k}]}; P).$
\item$\chi(\mathbf{x}, J, \mathbf{I}, M) = \chi(\mathbf{x}, J, \mathbf{I}, N)+\chi(\mathbf{x}, J, \mathbf{I}, P).$
\end{enumerate}
\end{corollary}

Now from Theorem \ref{thm1.vt} we prove the  additivity and reduction formula for
mixed multiplicities of maximal degrees  in  the following.

\begin{corollary}\label{co4.4a}   Assume that  the mixed multiplicity  of maximal degrees of $M$ with respect to
 $J, \bf I$ of the type $(k_0+1, {\bf k})$ is defined and let   $\bf x$ be a
joint reduction of the type $(\mathbf{k}, k_0+1)$ of $\mathbf{I}, J$ with respect to $M$. Set $ \Pi = \mathrm{Min}(A/\mathrm{Ann}(M)).$
 Then
 \begin{enumerate}[{\rm (i)}]
\item  $ e(J^{[k_0+1]},\mathrm{\bf I}^{[\mathrm{\bf k}]};M)= \sum_{\frak
p \in \Pi}\ell({M}_{\frak p})
 e(J^{[k_0+1]},\mathrm{\bf I}^{[\mathrm{\bf k}]};A/\frak p).$
\item $\chi(\mathbf{x}, J, \mathbf{I}, M) = \sum_{\frak
p \in \Pi}\ell({M}_{\frak p})
 \chi(\mathbf{x}, J, \mathbf{I}, A/\frak p).$
\end{enumerate}

 \end{corollary}
\begin{proof} (ii) follows from (i) and Theorem \ref{thm1.vt} (i). Now we prove (i).
 Let $$0 = M_0 \subseteq M_1 \subseteq M_2
  \subseteq\cdots \subseteq M_u=M$$ be a prime  filtration of $M$,
i.e., $M_{i+1}/M_i \cong A/P_i$ where
  $P_i$ is a prime ideal for all $0 \le i \le u-1$. Then $\Pi \subset \mathrm{Ass}(M)\subset \{P_0, \ldots, P_{u-1}\}
  \subset\mathrm{Supp}(M)$.
  By Theorem \ref{thm1.vt}(iii), we have \begin{equation}\label{eq2a} e(J^{[k_0+1]},\mathrm{\bf I}^{[\mathrm{\bf k}]};M)= \sum_{i=0}^{u-1}e(J^{[k_0+1]},\mathrm{\bf I}^{[\mathrm{\bf k}]}; A/P_i).\end{equation}

  Now we prove by induction on $|\bf k|$ that $e(J^{[k_0+1]},\mathrm{\bf I}^{[\mathrm{\bf k}]}; A/\frak p) = 0$
 for any prime ideal $\frak p \in \mathrm{Supp}(M)\setminus \mathrm{Min}(A/\mathrm{Ann}(M)).$ Indeed,
consider the case that $|{\bf k}| = 0$. If $I
\subset \frak
 p,$ then  $I(A/\frak p) = 0$, and so  $e(J^{[k_0+1]},\mathrm{\bf I}^{[\mathrm{\bf k}]};A/\frak p) =
 0$. Hence  we consider  the case that $\frak p\nsupseteq I.$
In this case, $\frak p \in \mathrm{Supp}(\overline{M}) \setminus
\mathrm{Min}(A/\mathrm{Ann}(\overline{M})),$ here $\overline{M} = M/0_M: I^\infty$  by \cite [Remark 3.3]{VT1}.  So $\dim A/\frak p < \dim \overline{M}.$  Since $\bigtriangleup^{(k_0,\mathrm{\bf
0})}P(n_0, {\bf n}, J,  \mathbf{I}, M)$ is a constant,  it follows that
$k_0+1 \ge \dim
\overline{M}$  by   Remark \ref{re2.00a}. Hence $\dim A/\frak p < k_0+1.$ Thus $\deg P(n_0, {\bf n}, J,  \mathbf{I}, A/\frak p) < k_0$ by   Remark \ref{re2.00a}. Therefore  $ \bigtriangleup^{(k_0,\mathrm{\bf
0})}P(n_0, {\bf n}, J,  \mathbf{I}, A/\frak p) = 0$.
 Hence
$e(J^{[k_0+1]},\mathrm{\bf I}^{[\mathrm{\bf 0}]}; A/\frak p)  = 0$ by Theorem \ref{thm1.vt} (i).

Consider the case that $|\mathbf{k}| > 0$. Recall that if $I
\subset \frak
 p,$ then $e(J^{[k_0+1]},\mathrm{\bf I}^{[\mathrm{\bf k}]};A/\frak p)  =
 0.$ Suppose that $\frak p\nsupseteq I.$
And without loss of generality, we can assume that $k_1>0$.  Then by \cite[Proposition
2.3]{VDT} (see \cite[Remark 1]{Vi}),  there exists $x\in I_1 \setminus
\frak p$ such that $x$ is a weak-(FC)-element of $\mathbf{I}, J$
with respect to $M$ and $A/\frak p$.   So by Proposition \ref{pro2.4b},
$$\bigtriangleup^{(k_0,\mathrm{\bf k})}P(n_0, {\bf n}, J,
\mathbf{I}, M)  =
\bigtriangleup^{(k_0,\mathrm{\bf k} -  \mathbf{e}_1)}P(n_0, {\bf n}, J,
\mathbf{I},  M\big/xM),$$
$$\bigtriangleup^{(k_0,\mathrm{\bf k})}P(n_0, {\bf n}, J,
\mathbf{I}, A/\frak p)  =
\bigtriangleup^{(k_0,\mathrm{\bf k} -  \mathbf{e}_1)}P(n_0, {\bf n}, J,
\mathbf{I},  A/(x)+\frak p).$$
Hence $\bigtriangleup^{(k_0,\mathrm{\bf k} -  \mathbf{e}_1)}P(n_0, {\bf n}, J,
\mathbf{I},  M\big/xM)$ and $\bigtriangleup^{(k_0,\mathrm{\bf k} -  \mathbf{e}_1)}P(n_0, {\bf n}, J,
\mathbf{I},  A/(x)+\frak p)$ are constants. Therefore the mixed multiplicities  of maximal degrees of $M/xM$ and $A/(x,\frak p)$ of the type $(k_0+1, \mathrm{\bf k} -  \mathbf{e}_1)$ are defined by Remark \ref{re2.00}, and moreover
$$e(J^{[k_0+1]},\mathrm{\bf I}^{[\mathrm{\bf k}]}; A/\frak p) = e(J^{[k_0+1]},\mathrm{\bf I}^{[\mathrm{\bf k}-\mathbf{e}_1]};A/(x)+\frak p)$$ by Note \ref{note2}. Let $\frak q$ be  an arbitrary
ideal in $\mathrm{Supp}(A/(x)+\frak p)$. Set $\bar A =
A/\mathrm{Ann}({M}).$ Denote by $\bar x, \bar {\frak p}, \overline
{(x)+\frak p}, \bar {\frak q}$ the images of $x, {\frak p},
{(x)+\frak p}, {\frak q}$ in $\bar A,$ respectively. Then since $x
\notin \frak p$ and $\frak p \in \mathrm{Supp}(M)\setminus
\mathrm{Min}(A/\mathrm{Ann}({M})),$  it follows that
$\mathrm{ht}_{\bar A} \overline {(x)+\frak p}
> 1.$ Consequence, $\mathrm{ht}_{\bar A/(\bar x)} [\overline {(x)+\frak p}/(\bar
x)]
> 0.$ So $\mathrm{ht}_{\bar A/(\bar x)} [\bar {\frak q}/(\bar
x)]
> 0.$  Hence
$\frak q \notin \mathrm{Min}(A/(x)+ \mathrm{Ann}({M})).$
 Note that
$\mathrm{Min}(A/(x)+\mathrm{Ann}({M})) =
\mathrm{Min}(A/\mathrm{Ann}({M}/xM))$ and $\frak q \in
\mathrm{Supp}({M}/xM).$  Thus $\frak q \in
\mathrm{Supp}({M}/xM)\setminus
\mathrm{Min}(A/\mathrm{Ann}({M}/xM)).$  Since $|\mathrm{\bf k}-
\mathbf{e}_1| = |\mathrm{\bf k}| -1, $
 applying the inductive hypothesis
for $M/x{M}$, we have $e(J^{[k_0+1]},\mathrm{\bf I}^{[\mathrm{\bf k}-\mathbf{e}_1]}; A/\frak q) = 0.$  So we obtain
\begin{equation}\label{eq2d}e(J^{[k_0+1]},\mathrm{\bf I}^{[\mathrm{\bf k}-\mathbf{e}_1]};A/\frak q) = 0\end{equation} for every $\frak q \in
\mathrm{Supp}(A/(x)+\frak p)$. Now, considering a prime filtration
of $A/(x) + \frak p$, by (\ref{eq2a}) and (\ref{eq2d}), we get
$e(J^{[k_0+1]},\mathrm{\bf I}^{[\mathrm{\bf k}-\mathbf{e}_1]};A/(x)+\frak p) = 0.$ So $e(J^{[k_0+1]},\mathrm{\bf I}^{[\mathrm{\bf k}]}; A/\frak p) = 0$ for every prime ideal $\frak
p \in  \mathrm{Supp}(M)\setminus \mathrm{Min}(A/\mathrm{Ann}({M})).$
The induction is complete. Therefore if  $P_i \notin
\Pi$, then
$e(J^{[k_0+1]},\mathrm{\bf I}^{[\mathrm{\bf k}]}; A/P_i) = 0.$  Note
that for any $\frak p \in \Pi$,  the number of times $\frak p$
appears in the sequence $P_0, \ldots, P_{u-1}$ is $\ell({M}_{\frak
p})$ since $\frak p \in \mathrm{Min}(A/\mathrm{Ann}({M})).$ Hence by
(\ref{eq2a}), we obtain $ e(J^{[k_0+1]},\mathrm{\bf I}^{[\mathrm{\bf k}]};M)= \sum_{\frak
p \in \Pi}\ell({M}_{\frak p})
 e(J^{[k_0+1]},\mathrm{\bf I}^{[\mathrm{\bf k}]};A/\frak p).$
  \end{proof}

\begin{remark}\label{re3.8} Set $ \Lambda =
\{\frak p\in \mathrm{Min}(A/\mathrm{Ann}(\overline{M}))\mid   \dim
A/\frak p \geq k_0 + 1 + |\mathbf{k}|\}$.  Then $\Lambda \subset \Pi.$ From the proof of Corollary \ref{co4.4a},  $$e(J^{[k_0+1]},\mathrm{\bf I}^{[\mathrm{\bf k}]};A/\frak p) = 0$$ if $\frak p \in \Pi \setminus \mathrm{Min}(A/\mathrm{Ann}(\overline{M})).$ In the case that $\dim A/\frak p < k_0 + 1 + |\mathbf{k}|$, then  $$\bigtriangleup^{(k_0,\mathrm{\bf
k})}P(n_0, {\bf n}, J,  \mathbf{I}, A/\frak p) = 0$$ since $\deg P(n_0, {\bf n}, J,  \mathbf{I}, A/\frak p) = \dim A/\frak p - 1$
by Remark \ref{re2.00a}.
 Hence $$e(J^{[k_0+1]},\mathrm{\bf I}^{[\mathrm{\bf k}]};A/\frak p) = 0$$  by Note \ref{note2}. Therefore  in Corollary \ref{co4.4a} we can replace the set $\Pi$ by the set $\Lambda$, i.e., $$e(J^{[k_0+1]},\mathrm{\bf I}^{[\mathrm{\bf k}]};M) = \sum_{\frak
p \in \Lambda}\ell({M}_{\frak p})
 e(J^{[k_0+1]},\mathrm{\bf I}^{[\mathrm{\bf k}]}; A/\frak p).$$
\end{remark}

 Finally, consider the case that $M$ has rank $r > 0$.    Recall that  $M$
has {\it rank} $r$ if $U^{-1}M$ (the localization of $M$ with
respect to $U$) is a free $U^{-1}A$-module of rank $r,$ here $U$
is the set of all  non-zero divisors of $A.$ Then since
$\mathrm{Ann}(U^{-1}M)=0,$  it follows that
 $\mathrm{Ann}(M) =0.$ Hence   $\Pi = \mathrm{Min}(A/\mathrm{Ann}(M))= \mathrm{Min}A.$
 Then $M_{\frak p}\cong (A_{\frak p})^r$ for all $\frak p \in \Pi$. So $\ell({M}_{\frak p})= r.\ell({A}_{\frak p})$.
 Therefore by Corollary \ref{co4.4a} (i),
$$e(J^{[k_0+1]},\mathrm{\bf I}^{[\mathrm{\bf k}]};M)=
\mathrm{rank}(M)\sum_{\frak p \in \Pi}\ell({A}_{\frak
p})e(J^{[k_0+1]},\mathrm{\bf I}^{[\mathrm{\bf k}]};A/\frak p)$$ and
   $e(J^{[k_0+1]},\mathrm{\bf I}^{[\mathrm{\bf k}]};A)= \sum_{\frak p
\in \Pi}\ell({A}_{\frak
p})e(J^{[k_0+1]},\mathrm{\bf I}^{[\mathrm{\bf k}]};A/\frak p).$
 Thus we get  $$e(J^{[k_0+1]},\mathrm{\bf I}^{[\mathrm{\bf k}]};M)=
 e(J^{[k_0+1]},\mathrm{\bf I}^{[\mathrm{\bf k}]};A)\mathrm{rank}(M).$$
 Consequently, we obtain the following corollary.

 \begin{corollary}\label{co4.4}  Let $M$ be an  $A$-module of positive rank.
 Assume that  the mixed multiplicity  of maximal degrees of $M$ with respect to
 $J, \bf I$ of the type $(k_0+1, {\bf k})$ is defined and $\bf x$ is a
joint reduction of $\mathbf{I}, J$ with respect to $M$ of the type $(\mathbf{k}, k_0 + 1)$.
Then  we
  have
\begin{enumerate}[{\rm (i)}]
\item  $e(J^{[k_0+1]},\mathrm{\bf I}^{[\mathrm{\bf k}]};M)=
e(J^{[k_0+1]},\mathrm{\bf I}^{[\mathrm{\bf
k}]};A)\mathrm{rank}(M).$
\item $\chi(\mathbf{x}, J, \mathbf{I},M)=
\chi(\mathbf{x}, J, \mathbf{I},A)\mathrm{rank}(M).$
\end{enumerate}
\end{corollary}

Returning to the original mixed multiplicities, we have the following notes.

 \begin{remark}\label{re4.1}  In the case  that $k_0 + |\mathbf{k}| = q - 1,$ from the above results we receive  respective results on  the original mixed
multiplicities as follows.
 \begin{enumerate}[{\rm (i)}]
 \item Theorem \ref{thm1.vt} (ii) also characterizes    the positivity of the original mixed
multiplicities of ideals in terms of minimal joint reductions. Note that this result has not been known in early works.
\item  From Theorem \ref{thm1.vt} (iii) we  obtain   the additivity  of  original mixed
multiplicities  \cite[Corollary 3.9]{VT1}.  Note that in \cite{VT1} the additivity  was showed via the
additivity and reduction formula which had to be proved by another
 approach. Moreover, it seems  that statements of the additivity for multiplicities of maximal degrees
 are shorter and more natural than  for  original mixed multiplicities.
\item From Corollary \ref{co4.4a} (i) and Remark \ref{re3.8} we get   the additivity and reduction formula  of original  mixed
multiplicities \cite[Theorem 3.2]{VT1}.
\item Finally, from  Corollary \ref{co4.4} (i)  we receive  \cite[Theorem 3.4]{VT2}.
\end{enumerate}
 \end{remark}

\noindent{\bf Acknowledgement:} {
\it This paper is accepted by Journal of Algebra and Its Applications,
 doi: ttps://doi.org/10.1142/S0219498821500262. Special thanks are due to the referee whose remarks substantially
improved the paper}.


\end{document}